\documentclass[final,onefignum,onetabnum,reqno]{siamart190516}

\usepackage{braket,amsfonts}
\usepackage{amsmath,amssymb}
\usepackage{array}

\usepackage[caption=false]{subfig}
\usepackage{pgfplots}

\newsiamthm{claim}{Claim}
\newsiamthm{ass}{Assumption}
\newsiamremark{remark}{Remark}
\newsiamremark{hypothesis}{Hypothesis}
\crefname{hypothesis}{Hypothesis}{Hypotheses}

\usepackage{algorithmic}

\usepackage{graphicx,epstopdf}
\usepackage{diagbox}

\Crefname{ALC@unique}{Line}{Lines}

\usepackage{amsopn}

\usepackage{xspace}
\usepackage{bold-extra}
\usepackage[most]{tcolorbox}
\usepackage{multirow}
\usepackage{algorithm}
\usepackage{algorithmic}

\colorlet{texcscolor}{blue!50!black}
\colorlet{texemcolor}{red!70!black}
\colorlet{texpreamble}{red!70!black}
\colorlet{codebackground}{black!25!white!25}

\lstdefinestyle{siamlatex}{%
  style=tcblatex,
  texcsstyle=*\color{texcscolor},
  texcsstyle=[2]\color{texemcolor},
  keywordstyle=[2]\color{texemcolor},
  moretexcs={cref,Cref,maketitle,mathcal,text,headers,email,url},
}

\tcbset{%
  colframe=black!75!white!75,
  coltitle=white,
  colback=codebackground, 
  colbacklower=white, 
  fonttitle=\bfseries,
  arc=0pt,outer arc=0pt,
  top=1pt,bottom=1pt,left=1mm,right=1mm,middle=1mm,boxsep=1mm,
  leftrule=0.3mm,rightrule=0.3mm,toprule=0.3mm,bottomrule=0.3mm,
  listing options={style=siamlatex}
}

\newtcblisting[use counter=example]{example}[2][]{%
  title={Example~\thetcbcounter: #2},#1}

\newtcbinputlisting[use counter=example]{\examplefile}[3][]{%
  title={Example~\thetcbcounter: #2},listing file={#3},#1}

\DeclareTotalTCBox{\code}{ v O{} }
{ 
  fontupper=\ttfamily\color{black},
  nobeforeafter,
  tcbox raise base,
  colback=codebackground,colframe=white,
  top=0pt,bottom=0pt,left=0mm,right=0mm,
  leftrule=0pt,rightrule=0pt,toprule=0mm,bottomrule=0mm,
  boxsep=0.5mm,
  #2}{#1}

\patchcmd\newpage{\vfil}{}{}{}
\begin{tcbverbatimwrite}{tmp_\jobname_header.tex}
\title{Adaptive Nonoverlapping Preconditioners for the Helmholtz Equation\thanks{Submitted to the editors May 1, 2025.
\funding{YY was supported by the National Natural Science Foundation of China (Project 12301497). ZZ was
partially supported by the National Natural Science Foundation of China (Project 12171406), the Hong Kong RGC grant (Project 17307921), and the Seed Funding Programme for Basic Research (HKU).}}}

\author{YI YU\thanks{School of Mathematics and Information Science, Guangxi University, P.\,R. China(\email{yiyu@gxu.edu.cn}).}
\and Marcus Sarkis\thanks{Mathematical Sciences, Worcester Polytechnic Institute, MA, USA(\email{msarkis@wpi.edu}).}
\and Guanglian Li\thanks{Department of Mathematics, The University of Hong Kong, Hong Kong, P.\,R. China(\email{lotusli@maths.hku.hk}).} \and 
Zhiwen Zhang\thanks{Department of Mathematics, The University of Hong Kong, Hong Kong, P.\,R. China (\email{zhangzw@hku.hk}).}}

\headers{NOSAS for the Helmholtz }{Yi Yu, Marcus Sarkis, Guanglian Li, Zhiwen Zhang}
\end{tcbverbatimwrite}
\input{tmp_\jobname_header.tex}

\ifpdf
\hypersetup{ pdftitle={NOSAS for Helmholtz } }
\fi

\begin{document}
\maketitle

\begin{tcbverbatimwrite}{tmp_\jobname_abstract.tex}
\begin{abstract}
The Helmholtz equation poses significant computational challenges due to its oscillatory solutions, particularly for large wavenumbers. 
Inspired by the Schur complement system for elliptic problems, this paper presents a novel substructuring approach to mitigate the potential ill-posedness of local Dirichlet problems for the Helmholtz equation. We propose two types of preconditioners within the framework of nonoverlapping spectral additive Schwarz (NOSAS) methods. The first type of preconditioner focuses on the real part of the Helmholtz problem, while the second type addresses both the real and imaginary components, providing a comprehensive strategy to enhance scalability and reduce computational cost. Our approach is purely algebraic, which allows for adaptability to various discretizations and heterogeneous Helmholtz coefficients while maintaining theoretical convergence for thresholds close to zero. Numerical experiments confirm the effectiveness of the proposed preconditioners, demonstrating robust convergence rates and scalability, even for large wavenumbers.
\end{abstract}


\end{tcbverbatimwrite}
\input{tmp_\jobname_abstract.tex}

\section{Introduction}
The time-dependent wave equation is a fundamental partial differential equation that arises in numerous scientific and engineering applications, including acoustics, electromagnetics, and seismic imaging. This paper focuses on preconditioning the Helmholtz equation, which corresponds to the time-harmonic case of the wave equation. Let $\Omega \subset \mathbb{R}^{2,3}$ be a bounded convex polygonal or polyhedral domain. We study the Helmholtz equation subject to impedance boundary conditions, formulated as:
\begin{equation}
    \label{hel_cont}
       \begin{split}
    -\Delta u-k^2u&=f~~~\mbox{in}~~~\Omega,\\
    \frac{\partial u}{\partial n}+iku&=g~~~\mbox{on}~~ \partial \Omega, 
     \end{split}
\end{equation}
where, without loss of generality, we assume the wavenumber $k \geq k_0 > 0$. The variational formulation of the Helmholtz equation is defined as follows: Given $f \in (H^1(\Omega))^\prime$ and $g \in H^{-1/2}(\Gamma)$, find $u \in H^1(\Omega)$ such that
\begin{equation}
\label{var_form}
    b(u,v)=l(v) \quad\text{for all } v\in H^1(\Omega),
\end{equation}
where 
\begin{align*}
  b(u,v)&:= \int_\Omega\nabla u\cdot\nabla\overline{ v}\mathrm{d}x-k^2\int_\Omega u\overline{v}\;\mathrm{d}x+ ik\int_{\partial \Omega} u\overline{ v} \;\mathrm{d}s\\
   l(v)&:=\int_\Omega f\overline{v}\;\mathrm{d}x+\int_{\partial \Omega} g\overline{v}\;\mathrm{d}s.    
\end{align*}
The existence and uniqueness of the weak solution of \cref{var_form} can be found in \cite{spence2014chapter}. 

Numerical schemes to solve the Helmholtz equation present significant challenges, especially when the wavenumber $k$ is large. Two main difficulties arise in the numerical solution of the Helmholtz equation. First, the solution $u$ exhibits highly oscillatory behavior, which requires fine meshes to accurately capture the oscillations. The number of degrees of freedom (DOFs) $N$ needed to resolve the solution typically scales as $N = k^d$, where $d$ is the spatial dimension. Furthermore, due to the "pollution effect" in the standard $H^1$-conforming variational formulation, achieving a numerical error that is comparable to the best approximation error requires even more stringent meshing conditions \cite{babuska1997pollution}. For example, when using piecewise linear finite element spaces, it is necessary that $k^2h$ be sufficiently small to maintain quasi-optimality, where $h$ denotes the mesh size. This requirement results in the generation of large linear systems that are computationally expensive to solve. Second, the resulting discrete systems are often strongly indefinite, leading to convergence issues for standard iterative solvers. Indefiniteness arises because the Helmholtz operator is not positive-definite, complicating the development of efficient solvers. Therefore, effective preconditioners are crucial to efficiently solving these large linear systems.

Over the years, several preconditioners have been proposed to address these challenges, each with its own set of advantages and limitations. The shifted Laplacian preconditioner \cite{erlangga2004class,MR3242986} is among the most commonly employed techniques. It involves solving a modified Helmholtz equation with a complex shift to enhance convergence properties. However, while this approach can accelerate convergence for small $k$, it still suffers from convergence difficulties for large wavenumbers. Sweeping preconditioners, as introduced in \cite{engquist2011sweeping, engquist2011sweeping2}, can achieve convergence in a few iterations, but their lack of parallelizability makes them less suitable for high-performance computing environments.

Domain Decomposition Methods (DDMs) divide a computational domain into smaller subdomains, allowing the problem to be solved in a parallelized manner. Traditional DDMs have demonstrated excellent performance and strong theoretical guarantees for elliptic preconditioners. However, they face significant challenges when applied to the Helmholtz equation. One of the main issues is that the local Dirichlet boundary problems can be ill-posed. Introducing a local impedance boundary condition, as suggested by \cite{kimn2007restricted}, can help alleviate this issue to some extent. However, achieving good convergence rates requires careful selection of the transmission conditions at the subdomain interfaces \cite{gander2019class, gander2006optimized}; see also \cite{MR4150274,MR4496963} for the analysis of such preconditioners. 

For overlapping domain decomposition methods, effective preconditioning generally requires a substantial overlap between subdomains. However, increasing the overlap also leads to greater computational costs and memory requirements, which can hinder scalability and efficiency, particularly in parallel computing environments. On the other hand, nonoverlapping methods have also been explored for the Helmholtz problem. These methods simplify the problem by eliminating overlap, which reduces communication costs between processors in parallel implementations. However, constructing an effective coarse space for nonoverlapping methods remains a considerable challenge. An effective coarse space must accurately capture the global behavior of the solution, which is difficult in Helmholtz problems because of their highly oscillatory nature and the indefiniteness of the operator.

Motivated by the success of using generalized eigenvalue problems to precondition elliptic equations with heterogeneous coefficients, we propose a novel substructuring approach and introduce two types of DDMs to address these challenges. A significant issue with traditional methods is the potential ill-posedness of the local Dirichlet boundary problem. To address this, we introduce a new iterative substructuring method, which is similar in concept to the Schur complement system used for elliptic problems. This new structure ensures the well-posedness of the local Dirichlet problems by incorporating the small-magnitude eigenvalues from each subdomain into the coarse problem. We note that because the partition of unity and overlapping are not used, our final results depend only on a few parameters, and there is no issue with respect to the redundancy of the coarse basis functions. We note that the use of generalized eigenvalue problems for the overlapping case was considered in \cite{MR4621835}. 

Another key challenge lies in constructing an effective coarse space for non-overlapping methods. To overcome this, we leverage the new substructuring approach along with the framework of Nonoverlapping Spectral Additive Schwarz (NOSAS) preconditioners \cite{yu2021additive, yu2024family} and propose two types of DDMs that tackle this issue. The proposed preconditioners aim to improve convergence by focusing on extracting eigenvalues that are either too small or too large. Initially, we consider only the real part of the Helmholtz problem, introducing two preconditioners: $P_{1}^{-1}$ and $P_{2}^{-1}$. While $P_{1}^{-1}$ provides a better convergence rate estimate, $P_{2}^{-1}$ is computationally less expensive, has better scalability, and achieves a comparable (albeit slightly worse) convergence rate. We then extend this approach to handle both the real and imaginary components of the Helmholtz system, resulting in two additional preconditioners: $P_{3}^{-1}$ and $P_{4}^{-1}$. Similarly, $P_{4}^{-1}$ is more cost-effective, provides better scalability, and offers convergence rates similar to those of $P_{3}^{-1}$. Importantly, all the generalized eigenvalue problems we consider involve real matrices with symmetric positive-definite right-hand sides, which allows us to compute the eigenvalues and eigenvectors using Cholesky decomposition rather than the generalized Schur factorization (QZ algorithm). Moreover, our construction is purely algebraic in nature, facilitating straightforward extension to other discretizations and the case of heterogeneous Helmholtz coefficients. This approach can be similarly applied to more general geometries and unstructured domain partitions. While the specific eigenvalue problems we solve do depend on the mesh and subdomain geometry, this dependence does not fundamentally alter the algebraic character of our method. Furthermore, the convergence theorems remain valid when the thresholds are close to zero. 

The remainder of this paper is organized as follows. In \Cref{section_2}, we introduce preliminary concepts related to discretization and domain decomposition notation. \Cref{section_3} presents a novel iterative substructuring scheme. In \Cref{section_4}, we introduce the first type of preconditioner, which focuses on the real part of the Helmholtz operator, demonstrating well-posedness and convergence when the threshold is chosen close to one. In \Cref{section_5}, we present the second type of preconditioner, designed to handle both real and imaginary components of the linear system, offering improved scalability when the imaginary part is globally assembled. Finally, \Cref{section_6} presents numerical experiments to validate the performance of the proposed preconditioners.

\section{Discretization and Domain Decomposition Notation}
\label{section_2}

To discretize the problem in \cref{hel_cont}, we first introduce a domain decomposition of $\Omega$. Let $\{\Omega_i\}_{i=1}^{N}$ be non-overlapping polygonal subdomains of size $O(H)$ that satisfy the following conditions:
\begin{equation*}
\overline{\Omega}=\bigcup_{i=1}^N\overline{\Omega}_i \hspace{10pt} \text{and} \hspace{10pt} \Omega_i\bigcap\Omega_j=\emptyset, \hspace{10pt}i\not=j.
\end{equation*}
Each subdomain is further partitioned into a union of shape-regular elements with mesh size $h$. This fine-scale partition of $\Omega$ is denoted by $\mathcal{T}_h$. We also define the interface and interior of each subdomain as $\Gamma_i$ and $I_i$, respectively. The global interface $\Gamma$ and the global interior $I$ are correspondingly defined as:
\begin{equation*}
  \Gamma_i = \partial \Omega_{i}, \quad    \Gamma=\bigcup_{i=1}^N \Gamma_{i},  \quad \mbox{and}\quad  I_i=\Omega_i,\quad I=\bigcup_{i=1}^N I_i.
\end{equation*}

The finite element space over $\mathcal{T}_h$ is denoted by $V_h(\Omega)$, and for convenience, we use the simplified notation $V_h$. Additionally, the local finite element space is defined as $V_h(\overline{\Omega}_i) = \{ v_h|_{\overline{\Omega}_i} : v_h \in V_h \}$.

The $H^1$-conforming Galerkin approximation to the problem \cref{var_form} seeks $u_h \in V_h$ such that:
\begin{equation}\label{eq:discrete-global}
    b(u_h,v_h)=l(v_h) \quad\text{for all } v_h\in V_h,
\end{equation}
and the corresponding linear system can be written as:
\begin{equation}
    \label{lin_form}
   Bu_h=l,
\end{equation}
where we use the same notation $u_h$ for both the finite element solution and its column vector of nodal values. Here, $B$ is the matrix representation of $b(\cdot, \cdot)$ in $V_h \times V_h$, and also denote $l$ as the column vector representation of $l(\cdot)$ in $V_h$.

Next, we consider the local assembly of the matrix $B$. We define the local sesquilinear form $b^{(i)}(\cdot, \cdot) : V_h(\Omega_i) \times V_h(\Omega_i) \to \mathbb{C}$ as:
\begin{equation*}
    b^{(i)}(u^{(i)},v^{(i)})=\int_{\Omega_i} (\nabla u^{(i)}\cdot\nabla \overline{v^{(i)}}-k^2u^{(i)}\overline{v^{(i)}})\mathrm{d}x +ik\int_{\partial\Omega_i\bigcap \partial\Omega}u^{(i)}\overline{v^{(i)}}\mathrm{d}s,
\end{equation*}
And the matrix form of $b^{(i)}(\cdot, \cdot)$ in $V_h(\Omega_i)$ is given by:
\begin{equation}\label{eq:localmat}
 B^{(i)}=\begin{bmatrix}
   B_{\Gamma\Gamma}^{(i)}       & B_{\Gamma I}^{(i)}    \\
    B_{I \Gamma}^{(i)}      & B_{I I}^{(i)}
\end{bmatrix},
\end{equation}
where the superscript denotes the subdomain number, and the subscripts denote the vectors associated with the nodal points in $\Gamma_i$ and $I_i$, respectively.

We note that \eqref{lin_form} can be expressed as:
\begin{equation*}
\begin{split}
 Bu_h=\begin{bmatrix}
   B_{\Gamma\Gamma}       & B_{\Gamma I}    \\
    B_{I \Gamma}      & B_{I I}
\end{bmatrix}u_h=\sum_{i=1}^N\begin{bmatrix}
  R^T_{\Gamma_i\Gamma} B_{\Gamma\Gamma}^{(i)} R_{\Gamma_i\Gamma}      & R_{\Gamma_i\Gamma}^TB_{\Gamma I}^{(i)} R_{I_iI}   \\
    R^T_{I_iI}  B_{I \Gamma}^{(i)} R_{\Gamma_i\Gamma}     & R^T_{I_iI}B_{I I}^{(i)}R_{I_iI}
\end{bmatrix}\begin{bmatrix}
  u_{\Gamma}    \\
    u_{I}     
\end{bmatrix}=\begin{bmatrix}
  l_\Gamma    \\
    l_I   
\end{bmatrix},
\end{split}
\end{equation*} 
where the extension operator $R_{\Gamma_i \Gamma}^{T} : V_h(\Gamma_i) \to V_h(\Gamma)$ is a zero extension from $\Gamma_i$ to $\Gamma$, with $V_h(\Gamma_i) := \{ v_{|{\Gamma_i}}; \forall v \in V_h(\Omega) \}$ and $V_h(\Gamma) := \{ v_{|{\Gamma}}; \forall v \in V_h(\Omega) \}$. Similarly, the extension operator $R_{I_i I}^{T} : V_h(I_i) \to V_h(I)$ is a zero extension from $I_i$ to $I$, with $V_h(I_i) := \{ v_{|{I_i}}; \forall v \in V_h(\Omega) \}$ and $V_h(I) := \{ v_{|{I}}; \forall v \in V_h(\Omega) \}$. The restriction operators $R_{\Gamma_i \Gamma} : V_h(\Gamma) \to V_h(\Gamma_i)$ and $R_{I_i I} : V_h(I) \to V_h(I_i)$ are the transposes of the corresponding extension operators, which restrict a nodal vector in a larger space to a nodal vector in a smaller space. An important property is that $R_{I_j I} R_{I_i I}^T$ is an identity matrix of size $n_i \times n_i$ if $i = j$, and a zero matrix of size $n_j \times n_i$ if $i \neq j$. Throughout, we use the superscript $T$ to denote the transpose and the superscript $*$ to denote the conjugate transpose.

We first introduce some standard notation for function spaces and their norms. For any $E \subset \Omega$, let $L^2(E)$ denote the standard Lebesgue space with the corresponding norm $\|\cdot\|_{L^2(E)}$, and let $H^1(E)$ denote the Sobolev space with the semi-norm $|\cdot|_{H^1(E)}$. We also define the following Helmholtz energy norm:
\begin{equation*}
\begin{split}
      ||v_h||_{\mathcal{H}}^2:= |v_h|^2_{H^1(\Omega)}+k^2||v_h||^2_{L^2(\Omega)}.
\end{split}    
\end{equation*}
When $\Omega_i$ is a subdomain of $\Omega$, we denote the corresponding norm on $\Omega_i$ by $\|\cdot\|_{\mathcal{H}, \Omega_i}$.

We consider the matrix representations of the norms $\|\cdot\|_{\mathcal{H}, \Omega_i}$ and $|\cdot|_{1, \Omega_i}$ as follows:
\begin{equation}
\label{a_i}
 H^{(i)}=\begin{bmatrix}
   H_{\Gamma\Gamma}^{(i)}       & H_{\Gamma I}^{(i)}    \\
    H_{I \Gamma}^{(i)}      & H_{I I}^{(i)}
\end{bmatrix},\quad\quad A^{(i)}=\begin{bmatrix}
   A_{\Gamma\Gamma}^{(i)}       & A_{\Gamma I}^{(i)}    \\
    A_{I \Gamma}^{(i)}      & A_{II}^{(i)}
\end{bmatrix},
\end{equation}
and the matrix representation of $\|\cdot\|_{\mathcal{H}}$ is given by:
\begin{equation*}
H=
\begin{bmatrix}
   H_{\Gamma\Gamma}       & H_{\Gamma I}    \\
    H_{I \Gamma}      & H_{II}
\end{bmatrix}=\sum_{i=1}^N\begin{bmatrix}
  R^T_{\Gamma_i\Gamma} H_{\Gamma\Gamma}^{(i)} R_{\Gamma_i\Gamma}      & R_{\Gamma_i\Gamma}^TH_{\Gamma I}^{(i)} R_{I_iI}   \\
    R^T_{I_iI}  H_{I \Gamma}^{(i)} R_{\Gamma_i\Gamma}     & R^T_{I_iI}H_{II}^{(i)}R_{I_iI}
\end{bmatrix}.
\end{equation*} 

Next, we recall the well-posedness of the discrete problem. Let $V_h$ be the finite element space on $\mathcal{T}_h$. We assume the following:

\begin{ass}[Discrete Inf-Sup Condition]
\label{ass_1}
Assume that $\Omega$ is a convex polygon/polyhedron domain,  for the finite element spaces $\mathcal{T}_h$, there exists a constant $\gamma>0$, such that the following discrete inf-sup condition holds:
\begin{equation*}
    \inf_{u\in V_h\backslash\{0\}}\sup_{v\in V_h\backslash\{0\}}\frac{ |b(u,v)|}{||u||_{\mathcal{H}}||v||_{\mathcal{H}}}\geq \gamma>0.
\end{equation*} 
\end{ass}

\begin{remark}
For linear finite element spaces, if $hk^2$ is sufficiently small, \cite[Proposition 8.2.7]{melenk1995generalized} shows that $\gamma = O\left(\frac{1}{1+k}\right)$, and the finite element solution $u_h$ satisfies:
\begin{equation*}
    ||u-u_{h}||_{\mathcal{H}} \lesssim  \inf_{v\in V_h(\Omega)} ||u-v||_{\mathcal{H}}\lesssim  hk (||f||_{L^2(\Omega)}+||g||_{L^2(\partial\Omega)}),
\end{equation*}
where $u$ is the exact solution of \cref{hel_cont}. The above mesh condition, along with these results, is typically referred to as $hk^2$-quasi-optimal. The $p$-finite element method and the $hp$-finite element methods were analyzed by Melenk and Sauter in \cite{melenk2010convergence} and \cite{melenk2011wavenumber}. For the $p$-finite element method, the best-known result is $hk^{\frac{p+1}{p}}$-quasi-optimal. The $hp$-finite element method achieves quasi-optimality if $kh/p$ is sufficiently small and $p \gtrsim 1 + \log(k)$ (i.e., $p$ grows logarithmically with $k$, with a fixed number of points per wavelength dependent on $p$). Such methods can be highly effective because the error converges exponentially with respect to the number of DOFs.

However, the above mesh condition can be too restrictive for practical computations. If we only require the well-posedness of a finite element solution, a less restrictive mesh condition can be used. For example, for FEM and CIP-FEM, if $hk^{\frac{2p+1}{2p}}$ is sufficiently small, the well-posedness of the discrete solution and the preasymptotic error estimate are guaranteed (see \cite{wu2014pre} and \cite{du2015preasymptotic}). Finally, depending on the type of finite element error (accuracy, data accuracy, quasi-optimality), different mesh conditions may be chosen. For a complete explanation, see \cite{pembery2020helmholtz}. In this paper, our preconditioner relies solely on the discrete inf-sup condition.
\end{remark}

Next, we introduce iterative substructuring methods by splitting the problem into $N$ local problems and a global interface problem. The classical approach to achieve this involves solving $N$ local homogeneous Dirichlet problems and one global inhomogeneous Neumann problem in $\Gamma$. The Neumann problem is obtained by eliminating interior unknowns at the $I_i$ nodes and introducing the Schur complement matrix $B_{\Gamma \Gamma}^{(i)} - B_{\Gamma I}^{(i)} (B_{I I}^{(i)})^{-1} B_{I \Gamma}^{(i)}$.

There are several issues with this approach when applied to the Helmholtz equation. The matrix $B_{I I}^{(i)}$ may be singular and/or indefinite. To address this, in \Cref{section_3}, we exclude a few DOFs in each subdomain. Another issue is that the Schur complement matrix can be singular and/or indefinite for floating subdomains. This problem is addressed in \Cref{section_4}.

\section{A New Iterative Substructuring Method} 
\label{section_3}

In this section, we introduce a novel substructuring approach that incorporates a global interface problem alongside $N$ local problems. The key idea is to extract small-magnitude eigenvalues from each subdomain $\Omega_i$ and incorporate them into the global problem to ensure the well-posedness of local problems. To achieve this, we consider the following generalized eigenvalue problem in each subdomain, seeking eigenpairs $(\mu^{(i)}_j, \phi_j^{(i)}) \in \mathbb{R} \times V_h(I_i)$ for $j=1, \ldots, n_i$, which satisfy:
\begin{equation}
\label{gen_eig_R00}
B_{II}^{(i)}\phi^{(i)}_j = \mu^{(i)}_jH_{II}^{(i)}\phi^{(i)}_j
\end{equation}
where the eigenvectors $\phi^{(i)}_j$ are orthonormal with respect to $H_{II}^{(i)}$, i.e., $(\phi^{(i)}_j)^T H_{II}^{(i)} \phi^{(i)}_k = \delta_{jk}$, $B_{II}^{(i)} \in \mathbb{R}^{n_i \times n_i}$ is symmetric, and $H_{II}^{(i)} \in \mathbb{R}^{n_i \times n_i}$ is symmetric and positive-definite, as defined in \cref{eq:localmat} and \cref{a_i}, respectively. Here, $n_i$ represents the DOFs in the interior of the subdomain $\Omega_i$. Consequently, the eigenvalues $\{\mu^{(i)}_j\}_{j=1}^{n_i}$ are real, non-increasingly ordered, and less than or equal to 1.

The rationale for choosing $H_{II}^{(i)}$ as the right-hand side matrix is its symmetric positive-definite nature, as well as the fact that the corresponding norm $\|\cdot\|_{\mathcal{H},\Omega_i}$ is used in the analysis. We select a small threshold value $\beta > 0$ and identify the eigenvalues whose magnitudes are smaller than $\beta$, that is, $|\mu^{(i)}_1| \leq |\mu^{(i)}_2| \leq \cdots \leq |\mu^{(i)}_{k_i}| < \beta$. We then construct $Q_S^{(i)} = [\phi^{(i)}_1, \ldots, \phi^{(i)}_{k_i}] \in \mathbb{R}^{n_i \times k_i}$ as the eigenspace corresponding to these small eigenvalues, the remaining eigenfunctions being denoted by $Q_L^{(i)} = [\phi^{(i)}_{k_i+1}, \ldots, \phi^{(i)}_{n_i}] \in \mathbb{R}^{n_i \times (n_i - k_i)}$. Here, the notation $S$ and $L$ represent "small" and "large" modes, respectively.

Accordingly, any $u^{(i)} \in V_h(\Omega_i)$ can be represented as:
\begin{equation*}
  u^{(i)} = \begin{bmatrix}
    u_{\Gamma_i} \\
    u_{I_i}
  \end{bmatrix} = \begin{bmatrix}
    I_{\Gamma_i} & 0 & 0 \\
    0 & Q_S^{(i)} &  Q_L^{(i)}
  \end{bmatrix} \begin{bmatrix}
    u_{\Gamma_i} \\
    \alpha_S^{(i)} \\
    \alpha_L^{(i)}
  \end{bmatrix}.
\end{equation*}
Here, $\alpha_S^{(i)} = [\alpha_{1}^{(i)}, \ldots, \alpha_{k_i}^{(i)}]^T \in \mathbb{C}^{k_i}$ and $\alpha_L^{(i)} = [\alpha_{k_i+1}^{(i)}, \ldots, \alpha_{n_i}^{(i)}]^T \in \mathbb{C}^{n_i - k_i}$. In this context, $\alpha_S^{(i)}$ and $\alpha_L^{(i)}$ represent the coefficients of $u_{I_i}$ with respect to the orthogonal bases $Q_S^{(i)}$ and $Q_L^{(i)}$, respectively.

Using the above notations, we can rewrite the linear system \cref{lin_form} as follows:
\begin{subequations}
    \begin{align}
       \label{1a}B_{\Gamma\Gamma}u_\Gamma + \sum_{i=1}^N R_{\Gamma_i\Gamma}^{T}B_{\Gamma I}^{(i)}Q_L^{(i)}\alpha_L^{(i)} + \sum_{i=1}^N R_{\Gamma_i\Gamma}^{T}B_{\Gamma I}^{(i)}Q_S^{(i)}\alpha_S^{(i)} &= l_\Gamma, \\
        \label{1b}B_{I\Gamma}u_\Gamma + \sum_{i=1}^N R_{I_iI}^{T}B_{I I}^{(i)}Q_L^{(i)}\alpha_L^{(i)} + \sum_{i=1}^N R_{I_iI}^{T}B_{I I}^{(i)}Q_S^{(i)}\alpha_S^{(i)} &= l_I.
     \end{align}
\end{subequations}

Next, we aim to reformulate the above system into one global interface problem and $N$ local problems. For the local problems, we consider:
\begin{equation*}
u_i = Q_L^{(i)}\big(Q_L^{(i)^T}B_{II}^{(i)}Q_L^{(i)}\big)^{-1}Q_L^{(i)^T} R_{I_i I}l_I.
\end{equation*}

In practice, the dimension $k_i$ of $Q_S^{(i)}$ is usually small, often zero, while $n_i - k_i$ is large. Therefore, it is computationally efficient to avoid directly calculating $Q_L^{(i)}$. Instead of directly computing $\big(Q_L^{(i)^T}B_{II}^{(i)}Q_L^{(i)}\big)^{-1}$, let us denote $B_L^{(i)} = Q_L^{(i)}\big(Q_L^{(i)^T}B_{II}^{(i)}Q_L^{(i)}\big)^{-1}Q_L^{(i)^T}$. To obtain $u_i$, we consider the following saddle point problem:
\begin{equation*}
\label{localmatrix}
\begin{bmatrix}
  B_{II}^{(i)}   & A_{II}^{(i)} Q_S^{(i)} \\
     Q_S^{(i)^T}A_{II}^{(i)^T} & 0 
\end{bmatrix}
\begin{bmatrix}u_i\\
 \lambda_S^{(i)}
\end{bmatrix} =
\begin{bmatrix} R_{I_iI} l_I\\
  {0}
\end{bmatrix}.
\end{equation*}
where $\lambda_S^{(i)} \in \mathbb{C}^{k_i}$ is a Lagrange multiplier. The well-posedness of this saddle point problem can be established by proving the uniqueness of the solution, utilizing the orthogonality of $Q_L^{(i)}$ and $Q_S^{(i)}$ with respect to the $H_{II}^{(i)}$-norm. Numerically, this problem can be solved using an $LDL^T$ factorization.

Next, we derive the global interface problem. 
Multiply $Q_L^{(i)^T} $ by \cref{1b} to obtain an explicit expression of $\alpha_{L}^{(i)}$ in terms of $u_i$ and $u_\Gamma^{(i)}$, for $1 \leq i \leq N$, and substitute in \cref{1a}. Also multiplying $Q_S^{(i)^T} R_{I_iI}$ in \cref{1b} again to eliminate each $\alpha_{L}^{(i)}$, the interface problem can be expressed as follows:
\begin{equation*}
     \begin{split}
       \sum_{i=1}^NR_{\Gamma_i\Gamma}^{T} \hat{B}_{\Gamma\Gamma}^{(i)}R_{\Gamma_i\Gamma}u_\Gamma+\sum_{i=1}^NR_{\Gamma_i\Gamma}^{T}B_{\Gamma I}^{(i)}Q_S^{(i)}\alpha_S^{(i)}&=l_\Gamma-\sum_{i=1}^NR_{\Gamma_i\Gamma}^{T}B_{\Gamma I}^{(i)} u_i,\\
Q_S^{(i)^T}B_{I\Gamma}^{(i)}R_{\Gamma_i\Gamma}u_\Gamma+Q_S^{(i)^T}B_{II}^{(i)}Q_S^{(i)}\alpha_S^{(i)}&= Q_S^{(i)^T}R_{I_iI}l_I, \quad \text{ for } 1\leq i\leq N
     \end{split}
 \end{equation*}
where we use the fact that $Q_L^{(i)^T} B_{II}^{(i)} u_i = 0$, and $\hat{B}_{\Gamma\Gamma}^{(i)} = B_{\Gamma\Gamma}^{(i)} - B_{\Gamma I}^{(i)} B_L^{(i)} B_{I\Gamma}^{(i)}$. The global problem can then be expressed in the following matrix form:
\begin{equation}
    \label{coarse_matrix}
    \sum_{i=1}^N
    \begin{bmatrix}
        \displaystyle{ R_{\Gamma_i\Gamma}^T \hat{B}_{\Gamma\Gamma}^{(i)}R_{\Gamma_i\Gamma}} & \displaystyle{R_{\Gamma_i\Gamma}^{T}B_{\Gamma I}^{(i)}Q_S^{(i)}R_S^{(i)}}\\[6pt]
        R_S^{(i)^T}Q_S^{(i)^T}B_{I\Gamma}^{(i)}R_{\Gamma_i\Gamma} & R_S^{(i)^T}Q_S^{(i)^T}B_{I I}^{(i)} Q_S^{(i)}R_S^{(i)}
    \end{bmatrix} 
    \begin{bmatrix}u_\Gamma\\[3pt]
    \alpha_S
    \end{bmatrix} =
    \sum_{i=1}^N
    \begin{bmatrix}l_\Gamma^{(i)}-R_{\Gamma_i\Gamma}^{T}{B}^{(i)}_{\Gamma I} u_i \\[6pt]
    R_S^{(i)^T}Q_N^{(i)^T}R_{I_iI}l_I
    \end{bmatrix},
\end{equation}
where $\alpha_S = [\alpha_S^{(1)}, \cdots, \alpha_S^{(N)}]^T$, $R_S^{(i)}$ is the restriction operator for selecting the local $i$-th index set $\alpha_S^{(i)}$ from all index sets $\alpha_S$, and $R_S^{(i)^T}$ is the corresponding transpose operator.

As mentioned earlier, the computation of $\hat{B}_{\Gamma\Gamma}^{(i)}$ using $B_L^{(i)}$ can be avoided by employing the following saddle point problem with Lagrange multipliers:
\begin{equation*}
   \begin{bmatrix}
     v_{\Gamma_i} \\
     \hat{v}_i \\
     \hat{\mu}_S^{(i)} \end{bmatrix}^T  
   \begin{bmatrix}
   B_{\Gamma \Gamma}^{(i)} & B_{\Gamma I}^{(i)} & 0  \\
   B_{I \Gamma}^{(i)}&  B_{I I}^{(i)} &  H_{II}^{(i)} Q_S^{(i)} \\  
   0      &  Q_S^{(i)^T}H_{II}^{(i)}  & 0
   \end{bmatrix}  \begin{bmatrix} u_{\Gamma_i} \\ \hat{u}_i \\
     \hat{\lambda}_S^{(i)} 
   \end{bmatrix}.
\end{equation*} 

This construction is purely algebraic. We now transform this formulation into a two-level nonoverlapping additive Schwarz method. We define the new interface space and new local spaces as follows:

\begin{definition}(New space)
Let
$$V_0 := V_h(\Gamma) \oplus \sum_{i=1}^N R_S^{(i)^T}\alpha_S^{(i)},\text{ and } \,\, V_i := \text{Range}(Q_L^{(i)}),$$
then the direct sum decomposition holds:
\[
 V_h(\Omega) = R_0^T V_0 \oplus R_1^T V_1 \oplus \dots R_N^T V_N.
\]
Here, $R_i^T : V_i \to V_h(\Omega)$ for $1 \leq i \leq N$ represents the zero extension to the nodal points in $\Omega_h \setminus I_i$. The extension operator $R_0^T : V_0 \to V_h(\Omega)$ is defined as:
\begin{equation*}
   R_0^T u_0 = \begin{bmatrix}
 I_\Gamma   & 0 \\
\displaystyle{\sum_{i=1}^N}-\!\!R_{I_iI}^TB_L^{(i)} B_{I\Gamma}^{(i)}R_{\Gamma_i\Gamma}  & \displaystyle{\sum_{i=1}^N}R_{I_iI}^TQ_S^{(i)}R_S^{(i)}
\end{bmatrix} u_0,
\end{equation*}
with its transpose operator $R_0: V_h(\Omega) \to V_0$ defined as:
\begin{equation*}
   R_0 u_h = \begin{bmatrix}
 I_\Gamma  & \displaystyle{\sum_{i=1}^N}-\!\!R_{\Gamma_i\Gamma}^TB_{\Gamma I}^{(i)}B_L^{(i)}  R_{I_iI}  \\
0 &  \displaystyle{\sum_{i=1}^N}R_S^{(i)^T}Q_S^{(i)^T} R_{I_iI}
\end{bmatrix} u_h,
\end{equation*}
where $I_\Gamma$ is the identity matrix with respect to $\Gamma$. We also define the local component of the coarse space $V_0^{(i)} = V_h(\Gamma_i) \oplus \alpha_S^{(i)}$.
\end{definition}

We follow the procedure of two-level additive Schwarz methods to construct the local and coarse solvers. First, we consider the local problem. We define the local sesquilinear form $b_{i}(\cdot,\cdot)$ on the local space $V_i$ as:
\begin{equation*}
\label{local_ses}
    b_{i}(u_i,v_i) = b(R_i^Tu_i, R_i^Tv_i) \hspace{15pt} \forall u_i, v_i \in V_i, \quad 1 \leq i \leq N.
\end{equation*}

Next, we define the projection-like operator $P_i: V_h \to V_h$ given by $P_i = R_{i}^{T}\tilde{P}_i$, where $\tilde{P}_i: V_h \to V_i$ is defined as the local solver for the following local problem:
\begin{equation}
\label{local_pb}
    b_i(\tilde{P}_iu_h, v_i) = b(u_h, R_{i}^{T}v_i) \hspace{25pt} \forall v_i \in V_i.
\end{equation}

The well-posedness of $\tilde{P}_i$ depends on the invertibility of the local sesquilinear form $b_{i}(\cdot,\cdot)$. It is important to note that since we choose $Q_L^{(i)}$ as the orthonormal basis of $H_{II}^{(i)}$, we have:
$$Q_L^{(i)^T}H_{II}^{(i)}Q_L^{(i)}=I^{(i)},\quad\text{and}\quad B_i = Q_L^{(i)^T}B_{II}^{(i)}Q_L^{(i)} = \text{diag}\{\mu_{k_{i+1}}, \dots, \mu_{n_i}\},$$
where $I^{(i)}$ is the identity matrix.

Next, we define the global sesquilinear form $b_0(\cdot,\cdot)$ on the interface space $V_0$ as:
\begin{equation}
\label{coarse_ses}
    b_0(u_0,v_0) = b(R_0^Tu_0, R_0^Tv_0) \hspace{15pt} \forall u_0, v_0 \in V_0,
\end{equation}
and we denote the corresponding matrix form as $B_0$, which is the left-hand side matrix in \cref{coarse_matrix}. We consider the projection-like operator $P_0 : V_0 \to V_h$ given by $P_0 = R_0^T \tilde{P}_0$, where $\tilde{P}_0 : V_h(\Omega) \to V_0$ is defined for the following global interface problem:
\begin{equation}
\label{coarse_pb}
  {b}_0(\tilde{P}_0u_h, v_0) = b(u_h, R_0^Tv_0) \hspace{25pt} \forall v_0 \in V_0.
\end{equation}

We note that the matrix form of the above coarse problem \cref{coarse_pb} is exactly \cref{coarse_matrix}. The well-posedness of $P_0$ depends on showing the inf-sup condition of the global sesquilinear form $b_0(\cdot,\cdot)$. To demonstrate this, we first define a norm $|||\cdot|||$ in the space $V_0$.  

For any $u_0 = \begin{bmatrix} u_\Gamma \\ \alpha_S \\ \end{bmatrix} \in V_0$, let $|||u_0||| = \|\mathcal{H}_0^T u_0\|_{\mathcal{H}}$, where the minimum $\mathcal{H}$-energy extension operator $\mathcal{H}_0^T : V_0 \to V_h(\Omega)$ is defined as:

\begin{equation*}
   \mathcal{H}_
0^T u_0 = \begin{bmatrix}
 I_\Gamma   & 0 \\
\displaystyle{\sum_{i=1}^N}-\!\!R_{I_iI}^TH_L^{(i)}H_{I\Gamma}^{(i)}R_{\Gamma_i\Gamma}  & \displaystyle{\sum_{i=1}^N}R_{I_iI}^TQ_S^{(i)}R_S^{(i)}
\end{bmatrix} u_0,
\end{equation*}
where $H_L^{(i)} = Q_L^{(i)} \big(Q_L^{(i)^T} H_{II}^{(i)} Q_L^{(i)}\big)^{-1} Q_L^{(i)^T}$. Since $\mathcal{H}_0^T$ is the minimum $\mathcal{H}$-energy extension, we have $\|\mathcal{H}_0^T u_0\|_{\mathcal{H}} \leq \|R_0^T u_0\|_{\mathcal{H}}$. The following theorem shows the stable decomposition of the coarse space and local spaces, as well as the inf-sup condition of $b_0(\cdot,\cdot)$, which guarantees the stability of the new system.

\begin{theorem}
\label{infsup_b_0}
Any $u_h\in V_h$ admits the unique decomposition in the form:
\begin{align*}
u_h=R_0^Tu_0+\displaystyle{\sum_{i=1}^N}R_i^Tu_i
\end{align*}
with $u_0\in V_0$ and $u_i\in V_i$ for $1\leq i\leq N$. 
Furthermore,  for any $v_0\in V_0$, we have 
 $$b_0(u_0,v_0)=b(u_h,R_0^Tv_0),$$
and  
\begin{equation*}
 \inf_{u_0\in V_0\backslash\{0\}}\sup_{v_0\in V_0\backslash\{0\}}\frac{|{b}_0(u_0,{v}_0)|}{|||u_0|||\hspace{3pt}|||{v}_0|||}\geq \gamma.
 \end{equation*}
\end{theorem}
\begin{proof}
Let $u_h=\begin{bmatrix}
  u_{\Gamma}    \\
    u_{I}     
\end{bmatrix}=\begin{bmatrix}u_\Gamma \\ \displaystyle{\sum_{i=1}^NR_{I_iI}^TQ_L^{(i)}\alpha_L^{(i)}+\sum_{i=1}^NR_{I_iI}^TQ_S^{(i)}\alpha_S^{(i)}}
\end{bmatrix} \in V_h(\Omega),$ then let ${u}_i=B_L^{(i)}(B_{I\Gamma}^{(i)}R_{\Gamma_i\Gamma}u_\Gamma+B_{II}^{(i)}R_{I_iI}u_I)\in V_i$  for $1\leq i \leq N$, and $u_0= \begin{bmatrix}
  u_{\Gamma}    \\
    \alpha_S
\end{bmatrix}\in {V}_0$. It is easy to check that $u_h=R_0^Tu_0+\displaystyle{\sum_{i=1}^N}R_i^T{u}_i.$ 

Next, we show the inf-sup condition of $b_0(\cdot,\cdot)$. Notice that by the definition of $b_0(\cdot,\cdot)$ in \cref{coarse_ses}, we have $b_0({u}_0,{v}_0)=b(R_0^T{u}_0,v_h),$ for any $v_h$ that satisfies $v_h=R_0^Tv_0+\displaystyle{\sum_{i=1}^N}R_i^T{v}_i$ with $v_i\in V_i$ for $1\leq i\leq N$.
Using the inf-sup condition of $b(\cdot,\cdot)$, we know that
for any $u=R_0^T{u}_0,$ there exists $v_h\in V_h,$ such that 
$$|b(R_0^T{u}_0,v_h)|\geq \gamma ||R_0^T{u}_0||_{\mathcal{H}}||v_h||_{\mathcal{H}}.
$$
Together with the definition of the norm $|||\cdot|||$, we have that for any ${u}_0$, there exists ${v}_0,$ such that:

\begin{equation*}
    | b_0({u}_0,{v}_0)|=| b(R_0^T{u}_0,v_h)|\geq \gamma ||R_0^T{u}_0||_{\mathcal{H}}||v_h||_{\mathcal{H}}\geq \gamma|||{u}_0|||\hspace{3pt}|||{v}_0|||.
\end{equation*}
\vspace{1pt}
\end{proof}

\Cref{infsup_b_0} implies that the well-posedness of the coarse problem \cref{coarse_pb} is a direct consequence of the well-posedness of the discrete problem \cref{eq:discrete-global}. Moreover, the inf-sup constant of the coarse problem is greater than or equal to the inf-sup constant of the discrete problem. Following the construction outlined above, we obtain an exact solver for the coarse problem and $N$ local problems, expressed as:
\begin{equation*}
  \label{first_pd}  B^{-1}=R_0^T({B}_0)^{-1}R_0+\displaystyle{\sum_{i=1}^N}{R}_i^TQ_P^{(i)}(B_i)^{-1}Q_P^{(i)^T}{R}_i
\end{equation*}
However, it is important to note that this direct solver has significant drawbacks in practical computation, primarily due to the computational cost associated with solving the interface problem. The size of the interface problem is determined by the number of DOFs in $V_h(\Gamma)$ plus $\sum_{i=1}^N k_i$.

In the next section, we will explore and discuss effective preconditioners for the interface problem $B_0$ using the NOSAS (Nonoverlapping Spectral Additive Schwarz) methods framework. We believe that this approach can be extended to other classes of Schwarz methods, providing efficient and scalable solvers for the Helmholtz problem.

\section{NOSAS Preconditioners for $\Re B_0$} 
\label{section_4}

In this section, we utilized the concept of Nonoverlapping Spectral Additive Schwarz (NOSAS) to develop a preconditioner for the newly introduced substructure. We design the preconditioner for $B_0 = \displaystyle{\sum_{i=1}^N} R_{\Gamma,S}^{(i)^T} B_0^{(i)} R_{\Gamma,S}^{(i)}$, where: \[
    B_0^{(i)} = 
    \begin{bmatrix}
        \hat{B}_{\Gamma\Gamma}^{(i)} & B_{\Gamma I}^{(i)} Q_S^{(i)} \\[6pt]
        Q_S^{(i)^T} B_{I\Gamma}^{(i)} & Q_S^{(i)^T} B_{I I}^{(i)} Q_S^{(i)}
    \end{bmatrix},
\]
and $R_{\Gamma,S}^{(i)} = \begin{bmatrix}
    R_{\Gamma_i, \Gamma} & 0 \\
    0 & R_S
\end{bmatrix}$ is the restriction operator, with $R_{\Gamma,S}^{(i)^T}$ representing its transpose. It is important to emphasize that $B_0$ is a complex matrix, necessitating separate treatment of its real and imaginary parts, which will be discussed in \Cref{section_5}. In this section, we focus on preconditioning the real component of $B_0$, represented by $\Re B_0$, while leaving the imaginary component $\Im B_0$ unchanged. We first introduce a preconditioner along with its theoretical proof of the convergence rate when thresholds close to zero. Additionally, we suggest a computationally more efficient algorithm that achieves similar numerical performance, albeit with a weaker bound in the convergence proof.

\subsection{Preconditioner $P_1^{-1}$}
Let us consider the following generalized eigenvalue problems for $i=1,\cdots,N$:
\begin{equation}
\label{gen_eig_R11}
\Re B_0^{(i)}\xi^{(i)}_j=\lambda^{(i)}_j H_{0}^{(i)}\xi^{(i)}_j
\hspace{30pt} (j=1,\cdots, N_i),
\end{equation}
where the eigenvectors $\{\xi^{(i)}_j\}$ are orthonormal with respect to $H_{0}^{(i)}$, and $N_i$ is the number of DOFs on $\Gamma_i$ combined with the dimensions of $Q_S^{(i)}$. The right-hand side positive-definite matrix $H_0^{(i)}$ corresponds to the minimum $\mathcal{H}$-energy extension $||\cdot||_{\mathcal{H}}$ norm as follows:
\begin{equation}
\label{H_0}
H_0^{(i)}=\begin{bmatrix}
   \hat{H}_{\Gamma\Gamma}^{(i)} & H_{\Gamma I}^{(i)}Q_S^{(i)} \\
   Q_S^{(i)^T}H_{I\Gamma}^{(i)} & Q_S^{(i)^T}H_{I I}^{(i)} Q_S^{(i)}
   \end{bmatrix},
\end{equation}
with $\hat{H}_{\Gamma\Gamma}^{(i)} = H_{\Gamma\Gamma}^{(i)} - H_{\Gamma I}^{(i)} H_L^{(i)} H_{I\Gamma}^{(i)}$. We note that $\forall u_0 \in V_0$, we have:
$$u_0^H\displaystyle{\sum_{i=1}^N R_{\Gamma,N}^{(i)^T}H_0^{(i)}R_{\Gamma,N}^{(i)}}u_0=||\mathcal{H}_
0^Tu_0||_{\mathcal{H}}^2.$$

We set a threshold $\eta \in (0, 1)$ and denote $Q_1^{(i)} = [\xi^{(i)}_1, \xi^{(i)}_2, \ldots, \xi^{(i)}_k]$, where the chosen eigenfunctions correspond to eigenvalues smaller than $1-\eta$ or greater than $1 + \eta$. Let $Q_1^{(i)^\perp} = [\xi^{(i)}_{k+1}, \ldots, \xi^{(i)}_{N_i}]$, which are the eigenfunctions corresponding to the eigenvalues between $1-\eta$ and $1 + \eta$. And define the projection operator $\Pi_1^{(i)}=Q_1^{(i)}(Q_1^{(i)^T}H_0^{(i)}Q_1^{(i)})^{-1}Q_1^{(i)^T}H_0^{(i)}=Q_1^{(i)}Q_1^{(i)^T}H_0^{(i)}$.

Next, note that $H_0^{(i)}$ is a dense matrix. To facilitate parallelization of the coarse problem, we consider using a symmetric positive-definite matrix $C_0^{(i)}$ with a block diagonal structure. The construction of $C_0^{(i)}$ follows two main principles:

1. Approximation to $\Re B_0^{(i)}$: The matrix $C_0^{(i)}$ should closely approximate $\Re B_0^{(i)}$ to minimize the number of eigenfunctions needed, implying that the generalized eigenvalues should be as close to 1 as possible.

2. Block Diagonal Structure: The matrix should be locally block diagonal with respect to the vertices and edges in each subdomain. This ensures that the globally assembled matrix retains a block-diagonal structure, facilitating parallel computation.

Thus, we propose the following construction for $C_0^{(i)}$:
\begin{equation*}
\label{hat_C_0}
 {C}_{0}^{(i)}=\begin{bmatrix}
   \hat{C}_{\Gamma\Gamma}^{(i)} & 0\\
   0 & Q_s^{(i)^T}H_{I I}^{(i)} Q_s^{(i)}
   \end{bmatrix},
\end{equation*}
where $\hat{C}_{\Gamma\Gamma}^{(i)}$ is the block-diagonal version of $\hat{H}_{\Gamma\Gamma}^{(i)}$ by breaking the connection between subdomain vertices and edges.

Then, we should also consider the following generalized eigenvalue problem:
\begin{equation}
\label{gen_eig_R12}
(Q_{1}^{(i)^\perp})^T C_0^{(i)}Q_{1}^{(i)^\perp}\phi^{(i)}_j=\lambda^{(i)}_j (Q_{1}^{(i)^\perp})^TH_{0}^{(i)}Q_{1}^{(i)^\perp}\phi^{(i)}_j,
\end{equation}
where the eigenvectors $\phi^{(i)}_j$ are orthonormal with respect to $(Q_{1}^{(i)^\perp})^TH_{0}^{(i)}Q_{1}^{(i)^\perp}$. We denote $Q_2^{(i)} = [Q_1^{(i)^\perp} \phi^{(i)}_1, Q_1^{(i)^\perp} \phi^{(i)}_2, \ldots, Q_1^{(i)^\perp} \phi^{(i)}_s]$, where the chosen eigenfunctions correspond to the eigenvalues that are smaller than $1-\eta$ or greater than $1 + \eta$. Define the projection operator $\Pi_2^{(i)}=Q_2^{(i)}Q_2^{(i)^T}H_0^{(i)}.$


Notice that $\Pi_1^{(i)} \Pi_2^{(i)} = 0$. Let us define: $\Pi_{Re}^{(i)}=\Pi_1^{(i)}+\Pi_2^{(i)}=Q_{Re}^{(i)}Q_{Re}^{(i)^T}H_0^{(i)}$, where $Q_{Re}^{(i)} = [Q_1^{(i)}, Q_2^{(i)}]$. Then it is straightforward to see that:
\begin{equation*}
(\Pi_{Re}^{(i)})^TH_0^{(i)}(I-\Pi_{Re}^{(i)})=0.
\end{equation*}

Let us define our local preconditioned sesquilinear form ${b}_{P_1}^{(i)}(\cdot,\cdot) : V_0^{(i)} \times V_0^{(i)} \to \mathbb{C}$ as follows:
\begin{equation}
\label{sesq_local1}
    \begin{split}
     {b}_{P_1}^{(i)}(u_{0}^{(i)},v_{0}^{(i)})&=v_{0}^{(i)^H}\Big(\Re B_0^{(i)}+\textbf{i}\Im B_{0}^{(i)}-(I-\Pi_{Re}^{(i)^T})\Re B_0^{(i)}(I-\Pi_{Re}^{(i)})+(I-\Pi_{Re}^{(i)^T})C_{0}^{(i)}(I-\Pi_{Re}^{(i)})\Big)u_{0}^{(i)}\\
&=v_{0}^{(i)^H}\Big(C_0^{(i)}+\textbf{i}\Im B_{0}^{(i)}-D_{Re}^{(i)}\Pi_{Re}^{(i)}-\Pi_{Re}^{(i)^T} D_{Re}^{(i)}(I-\Pi_{Re}^{(i)})\Big)u_{0}^{(i)},
\end{split}
\end{equation}
where $D_{Re}^{(i)}=C_0^{(i)}-\Re B_0^{(i)}$.

The global sesquilinear form ${b}_{P_1}(\cdot, \cdot) : V_0 \times V_0 \to \mathbb{C}$ is then defined as:
\begin{equation*}
\label{sesq_1}
{b}_{P_1}(u_0,v_0)=\sum_{i=1}^N{b}_{P_1}^{(i)}(R_{\Gamma,N}^{(i)}u_0,R_{\Gamma,N}^{(i)}v_0).
\end{equation*}
Consequently, we denote the matrix representation of the global sesquilinear form ${b}_{P_1}(\cdot, \cdot)$ as ${B}_{P_1}$.

Together with the local problems, we define the first preconditioner as:
\begin{equation}
\label{P_ad1}
P_{1}^{-1}=R_0^TB_{P_1}^{-1}R_0+\displaystyle{\sum_{i=1}^N}{R}_i^TQ_P^{(i)}(B_i)^{-1}Q_P^{(i)^T}{R}_i.
\end{equation}

\begin{remark}
Apart from some specialized discretization methods, we emphasize that for $B_0^{(i)}$, 
while $B_{\Gamma I}^{(i)}$, $B_{II}^{(i)}$, and $B_{I\Gamma}^{(i)}$ are real matrices, 
the imaginary part of $\hat{B}_{\Gamma\Gamma}^{(i)}$ is nonzero only when the subdomain 
$\Omega_i$ touches the impedance boundary. We avoid complex generalized eigenvalue problems because the generalized Schur form (QZ algorithm) for a complex matrix may not be diagonalizable.
\end{remark}

\begin{remark}
\label{remark_4.2}
The orthonormal property of the generalized eigenfunctions $\xi_j^{(i)}$ with respect to $H_0^{(i)}$ ensures that $Q^{(i)^T} H_0^{(i)} Q^{(i)}$ is the identity matrix. Furthermore, let $Q_{Re}^{(i)^\perp}$ be the orthogonal complement of $Q_{Re}^{(i)}$ with respect to $H_0^{(i)}$. We have the following properties:
\begin{equation*}
 (1-\eta) v_{\Gamma_i}^T H_{0}^{(i)}v_{\Gamma_i}\leq v_{\Gamma_i}^T\Re B_0^{(i)}v_{\Gamma_i} \leq  (1+\eta) v_{\Gamma_i}^T H_{0}^{(i)}v_{\Gamma_i} \quad \forall v_{\Gamma_i} \in \mbox{Range}(Q_{Re}^{(i)^{\perp}}).
\end{equation*}
The above properties also hold when replacing $\Re B_0^{(i)}$ with $C_0^{(i)}$.
\end{remark}

\begin{remark}
The sesquilinear form \cref{sesq_local1} is in the form of a block diagonal part plus two low-rank perturbation parts. Hence, the globally assembled matrix also takes the form of a block diagonal plus some low-rank perturbation parts. Since $\Im B_0^{(i)}$ appears in the subdomains that touch the impedance boundary condition, the largest block is associated with all the DOFs on the impedance boundary, while the rest of the block is associated with all the DOFs on vertices and edges. The parallelization property is obtained using the Woodbury matrix identity, which we will illustrate at the end of this subsection.
\end{remark}

We note that the size of the global problem in the above method is twice the total number of selected eigenfunctions from both \cref{gen_eig_R11} and \cref{gen_eig_R12}, since the sesquilinear form \cref{sesq_local1} contains two low-rank perturbation parts. Moreover, we need to find all the eigenfunctions of \cref{gen_eig_R11}. We propose a cheaper approach that does not require finding all eigenfunctions; instead, only a small number of eigenvalues are needed. Furthermore, the size of the global problem is just the number of selected eigenfunctions.

\subsection{Preconditioner $P_2^{-1}$}
Instead of considering two generalized eigenvalue problems, we consider the following generalized eigenvalue problem locally:
\begin{equation}
\label{gen_eig_R}
\Re B_0^{(i)}\xi^{(i)}_j=\lambda^{(i)}_j C_{0}^{(i)}\xi^{(i)}_j
\hspace{30pt} (j=1,\cdots, N_i),
\end{equation}
where the eigenvectors $\xi^{(i)}_j$ are orthonormal with respect to $C_{0}^{(i)}$.

We choose the eigenvectors corresponding to eigenvalues smaller than $1-\eta$ and larger than $1 + \eta$ to construct the eigenfunction space $Q = [\xi^{(i)}_1, \cdots, \xi_{K_i}^{(i)}] \in \mathbb{R}^{N_i \times K_i}$. We then define the local projection operators $
\Pi_Q^{(i)}=Q^{(i)}Q^{(i)^T}C_{0}^{(i)}$.

Next, we define the local sesquilinear form ${b}_{P_2}^{(i)}(\cdot, \cdot) : V_0^{(i)} \times V_0^{(i)} \to \mathbb{C}$ as:
\begin{equation*}
    \begin{split}
     {b}_{P_2}^{(i)}(u_{0}^{(i)},v_{0}^{(i)})&=v_{0}^{(i)^H}\big(\Pi_Q^{(i)^T}\Re B_0^{(i)}\Pi_Q^{(i)}+(I-\Pi_Q^{(i)^T})C_{0}^{(i)}(I-\Pi_Q^{(i)})+\textbf{i}\Im B_{0}^{(i)} \big)u_{0}^{(i)}\\
     &=v_{0}^{(i)^H}\big(C^{(i)}_{0}+\textbf{i}\Im B_{0}^{(i)}\!-C_{0}^{(i)}Q^{(i)}D^{(i)}Q^{(i)^T}\!\!C_{0}^{(i)}\big)u_{0}^{(i)},
\end{split}
\end{equation*}
where $D^{(i)} = \text{diagonal}(1 - \lambda_1^{(i)}, \cdots, 1 - \lambda_{K_i}^{(i)}) \in \mathbb{R}^{K_i \times K_i}$. The global sesquilinear form ${b}_{P_2}(\cdot, \cdot)$ is then defined as the sum of all local sesquilinear forms. Then, we obtain the resulting preconditioner as follows:
\begin{equation}
\label{P_ad2}
P_{2}^{-1}=R_0^T(B_{P_2})^{-1}R_0+\displaystyle{\sum_{i=1}^N}{R}_i^TQ_P^{(i)}(B_i)^{-1}Q_P^{(i)^T}{R}_i.
\end{equation}

\subsection{Scalability of the preconditioner}
 We show the scalability of $B_{P_1}^{-1}$. The scalability of $B_{P_2}^{-1}$ is obtained in a similar way. Let $\displaystyle{
    C\!=\!\sum_{i=1}^N} R_{\Gamma,s}^{(i)^T}\big({C}_0^{(i)} +\textbf{i}R_{\Gamma,0}^{(i)^T}\Im B_{0}^{(i)}R_{\Gamma,0}^{(i)}\big)R_{\Gamma,s}^{(i)},$  $U_1\!=\! \displaystyle{\sum_{i=1}^N}R_{\Gamma,s}^{(i)^T}{D}_{Re}^{(i)} Q_{Re}^{(i)}R_{\lambda_i}$, $V_1=\displaystyle{\sum_{i=1}^NR_{\lambda_i}^TQ_{Re}^{(i)^T}H_0^{(i)}R_{\Gamma,s}^{(i)}}$, $U_2\!=\! \displaystyle{\sum_{i=1}^N}R_{\Gamma,s}^{(i)^T}{H}_0^{(i)} Q_{Re}^{(i)}R_{\lambda_i}$, $V_2=\displaystyle{\sum_{i=1}^NR_{\lambda_i}^TQ_{Re}^{(i)^T}D_{Re}^{(i)}(I^{(i)}-Q_{Re}^{(i)}Q_{Re}^{(i)^T}H_0^{(i)})R_{\Gamma,s}^{(i)}}$,  where $R_{\lambda_i}$ is choose the eigenfunctions in i-th subdomain from all chosen eigenspace. Then the matrix of $b_{P_1}(\cdot,\cdot)$ can be written as:
\begin{equation*}   
B_{P_1}=C-U_1V_1-U_2V_2.
\end{equation*}

Since $C^{-1}$ can be computed in parallel, we employ the Woodbury matrix identity (see Appendix \cref{woodbury}) to obtain the explicit expression:
\begin{equation*}
    {B}_{P_1}^{-1}=C^{-1}+C^{-1}[U_1,U_2]M^{-1}\begin{bmatrix}
V_1\\
V_2
   \end{bmatrix}C^{-1},
\end{equation*}
where $M=I-
\begin{bmatrix}
V_1\\
V_2
   \end{bmatrix}C^{-1}[U_1,U_2]$ is a dense matrix.

\begin{remark}
\label{remark4.4}
For the Helmholtz equation with Dirichlet or Neumann boundary conditions, the imaginary part of $B_{\Gamma\Gamma}^{(i)}$ vanishes. For some special discretizations, such as Hybrid Discontinuous Galerkin (HDG) with real penalty parameters, the imaginary part of $B_{\Gamma\Gamma}^{(i)}$ can be constructed as block diagonal. In such cases, $C$ becomes a block diagonal matrix with each block associated only with vertices and edges, thus $C^{-1}$ has an excellent parallel structure. For cases where $\Im B_0^{(i)}$ is a globally assembled matrix, the largest block is associated with all the DOFs on the impedance boundary, while the rest of the smaller blocks are associated with vertices and edges, therefore $C^{-1}$ can also be computed in parallel. Finally, for $B_{P_1}^{-1}$ the size of the global matrix $M$ is twice the number of selected eigenfunctions, while for $B_{P_2}^{-1}$, the size of the global matrix equals the number of selected eigenfunctions.
\end{remark}

We also consider the computational cost associated with each iteration of the algorithm. As for all our proposed preconditioners using $R_0^T$ and $R_0$ for the coarse problem,  it is important to note that the explicit construction of $R_0^T$ and $R_0$ is unnecessary, as their actions are equivalent to solving $N$ local problems.  One of the advantages of using $R_0^T$ and $R_0$ is that, if the residual satisfies $r = R_0^T r_0$, and given that $Q_L^{(i)^T} R_i B r = 0$, we can deduce:
\[
P^{-1} B r = R_0^T B_{P}^{-1} R_0 B r.
\]
This approach is similar to the "residual correction" method discussed in \cite{yu2024family}, which is used for harmonic extension in elliptic problems. Therefore, instead of solving a large preconditioned system for B, we solve a smaller preconditioned system for $B_0$.
The proposed algorithm for the numerical implementation of the preconditioner described in \cref{P_ad1} is as follows:
\begin{algorithm}
\caption{NOSAS Preconditioner for the Helmholtz Equation}
\label{Alg_1}
\begin{algorithmic}
\STATE \textbf{Input:} Matrices $R_i$, $B_{\Gamma I}^{(i)}$, $B_0$, preconditioner $B_{P}$, right-hand side vector $b$.
\STATE \textbf{Output:} Solution $u_h$.

\STATE \textbf{Step 1:} Solve the $N$ local problems to obtain solutions $u_i$ for $1 \leq i \leq N$. Then, multiply each local solution by the corresponding local matrix $-B_{\Gamma I}^{(i)}$ to construct $r_0$, where $r_0=R_0b$.

\STATE \textbf{Step 2:} Solve the preconditioned system:
\[
B_{P}^{-1} B_0 u_0 = B_{P}^{-1} r_0.
\]

\STATE \textbf{Step 3:} Upon obtaining the coarse solution $u_0$ by solving the preconditioned system, compute $R_0^T u_0$, which is equivalent to solving $N$ local problems. The final solution is then given by $u_h = R_0^T u_0 + \displaystyle{\sum_{i=1}^N} R_i^T u_i$.
\end{algorithmic}
\end{algorithm}

In particular, local problems are solved only twice, before iteration to obtain $u_i$ for $1\leq i\leq N$ and $R_0b$, and after iteration to construct $R_0^Tu_0$. During iteration, we solve the preconditioned system of the smaller matrix $B_0$ using an iterative method, rather than the entire matrix $B$. Here, $B_0$ represents the Schur complement for the Helmholtz equation, and $B_P$ denotes the preconditioner for $B_0$. Importantly, no local problems need to be solved during the iteration phase.

\subsection{Well-posedness of the Coarse Problem When $\eta$ Is Close to zero}
The well-posedness of the coarse problem is closely related to the inf-sup condition of the sesquilinear form $b_{P_1}(\cdot,\cdot)$. The following theorem establishes the inf-sup condition of $b_{P_1}(\cdot,\cdot)$ when $\eta$ is close to zero.
\begin{theorem}
\label{inf_sup_tilde_b0}
If $\eta$ is chosen such that $\gamma_1:=\gamma-2\eta>0$, the global sesquilinear form $b_{P_1}(\cdot,\cdot)$ satisfy the following inf-sup condition:
\begin{equation*}
\label{inf_sup_gamma1}
 \inf_{{u}_0\in {V}_0\backslash\{0\}}\sup_{{v}_0\in {V}_0\backslash\{0\}}\frac{|b_{P_1}({u}_0,{v}_0)|}{|||{u}_0|||\hspace{3pt}|||{v}_0|||}\geq \gamma_1.
\end{equation*} 
\end{theorem}
\begin{proof}
We first show that 
\begin{equation}
\label{diffb0bp1}
    |b_0(u_0,v_0)-b_{P_1}(u_0,v_0)| \leq 2\eta|||{u}_0|||\hspace{3pt}|||{v}_0|||.
\end{equation}

Using the definition of sesquilinear form \cref{sesq_local1}, we note that
\begin{equation*}
\begin{split}
     |b_0(u_0,v_0)&-b_{P_1}(u_0,v_0)|=\Big|\sum_{i=1}^N v_0^HR_{\Gamma,N}^{(i)^T}(I-\Pi_{Re}^{(i)^T})(C_{0}^{(i)}-\Re B_0^{(i)})(I-\Pi_{Re}^{(i)}) R_{\Gamma,N}^{(i)}u_0\Big|\\
     &=\Big|\sum_{i=1}^N v_0^HR_{\Gamma,N}^{(i)^T}(I-\Pi_{12}^{(i)^T})(C_{0}^{(i)}-H_0^{(i)}+H_0^{(i)}-\Re B_0^{(i)})(I-\Pi_{Re}^{(i)}) R_{\Gamma,N}^{(i)}u_0\Big|\\
     &\leq \textcircled{1}+\textcircled{2},
\end{split}
\end{equation*}
where $\textcircled{1}:= \Big|\displaystyle{\sum_{i=1}^N} v_0^HR_{\Gamma,N}^{(i)^T}(I-\Pi_{12}^{(i)^T})(C_{0}^{(i)}-H_0^{(i)})(I-\Pi_{Re}^{(i)}) R_{\Gamma,N}^{(i)}u_0\Big|$ and $\textcircled{2}:=\Big|\displaystyle{\sum_{i=1}^N} v_0^HR_{\Gamma,N}^{(i)^T}(I-\Pi_{12}^{(i)^T})(H_0^{(i)}-\Re B_0^{(i)})(I-\Pi_{Re}^{(i)}) R_{\Gamma,N}^{(i)}u_0\Big|$.

We first show that $\textcircled{1}\leq \eta|||u_0|||\,|||v_0|||$. Consider $V_h(\Gamma_i)=Q_{Re}^{(i)}\oplus Q_{Re}^{(i)^\perp}$, which are orthogonal with respect to $H_0^{(i)}$. And we denote $R_{\Gamma,N}^{(i)}u_0=u_{0}^{(i)}:=u_1+u_2$, where $u_1\in Q_{Re}^{(i)}$ and $u_2\in Q_{Re}^{(i)^\perp}$. Similarly,  $R_{\Gamma,N}^{(i)}v_0=v_{0}^{(i)}:=v_1+v_2$ where $v_1\in Q^{(i)}$ and $v_2\in Q^{(i)^\perp}$. Let $\varphi_1,\cdots,\varphi_{K}$ be any set of the normalized basis vectors of $Q_{Re}^{(i)}$ with respect to $H_0^{(i)}$, and $\varphi_{K+1},\cdots,\varphi_{n}$ be the normalized basis of $Q_{Re}^{(i)^\perp}$ with respect to $H_0^{(i)}$ and orthogonal with $C_0^{(i)}$. i.e.,
\begin{equation*}[\varphi_{K+1},\cdots,\varphi_{n}]^TH_0^{(i)}[\varphi_{K+1},\cdots,\varphi_{n}]=I,\,\,[\varphi_{K+1},\cdots,\varphi_{n}]^TC_0^{(i)}[\varphi_{K+1},\cdots,\varphi_{n}]=\Lambda,
\end{equation*}
where $\Lambda$ is a diagonal matrix with eigenvalue  $\lambda_{K+1},\cdots, \lambda_{n}$ on the diagonal.  Using \cref{remark_4.2}, we note that $1-\eta\leq \lambda_{K+1}\leq \cdots\leq \lambda_{n}\leq (1+\eta)$.  

Let $u_0^{(i)}=\alpha_1\varphi_1+\alpha_2\varphi_2+\cdots \alpha_n\varphi_n$ and $v_0^{(i)}=\beta_1\varphi_1+\beta_2\varphi_2+\cdots \beta_n\varphi_n$. Then, we have
\begin{equation*}
\begin{split}
    \textcircled{1}&={\sum_{i=1}^N}\Big|v_{0}^{(i)^H}(I-\Pi_Q^{(i)^T})(C_{0}^{(i)}-H_0^{(i)})(I-\Pi_Q^{(i)})u_{0}^{(i)}\Big|\\
    &={\sum_{i=1}^N}|(1-\lambda_{K+1})\alpha_{K+1}\beta_{K+1}+\cdots +(1-\lambda_{n})\alpha_{n}\beta_{n}|\\
    &\leq {\sum_{i=1}^N}|\alpha_{K+1}^2+\cdots+\alpha_{n}^2|^{1/2}|(1-\lambda_{K+1})^2\beta_{K+1}^2+\cdots+(1-\lambda_{n})^2\beta_{n}^2|^{1/2}\\
    &\leq \eta {\sum_{i=1}^N}|\alpha_{K+1}^2+\cdots+\alpha_{n}^2|^{1/2}|\beta_{K+1}^2+\cdots+\beta_{n}^2|^{1/2}\\
    &=\eta {\sum_{i=1}^N}(u_2^T H_0^{(i)}u_2)^{1/2}(v_2^T H_0^{(i)}v_2)^{1/2}\leq \eta|||u_0|||\,|||v_0|||.
    \end{split}
\end{equation*}

Similarly, we can show that $\textcircled{2} \leq \eta |||u_0||| \, |||v_0|||$, except that we use $\phi_{K+1}, \ldots, \phi_n$ as the normalized basis of $Q_{Re}^{(i)^\perp}$ with respect to $\Re B_0^{(i)}$ and $H_0^{(i)}$.

Next, using the inf-sup condition of $b_0(\cdot,\cdot)$ in \cref{infsup_b_0}, we have:
$$|{b}_0({u}_0,{v}_0)|\geq \gamma |||{u}_0|||\hspace{3pt}|||{v}_0|||.$$

Combining the above result with the triangle inequality, we obtain the desired result.
\end{proof}

The next theorem provides the key bounds for the preconditioned system of the Helmholtz equation.

\begin{theorem}
\label{I-P_B} 
For any $v_h\in V_h$, we have 
$$||(I-P_{1}^{-1}B)v_h||_{\mathcal{H}}\leq C(\eta,\gamma_1)||v_h||_{\mathcal{H}},$$
where $C(\eta,\gamma_1)=\frac{2\eta}{\gamma_1},$ and $I:V_h\to V_h$ is the identity mapping matrix.
\end{theorem}

\begin{proof}
First, notice that  $I=R_0^T({B}_0)^{-1}R_0B+\displaystyle{\sum_{i=1}^N}{R}_i^TQ_P^{(i)}(B_i)^{-1}Q_P^{(i)^T}{R}_iB$, so 
\begin{equation}
\label{I-PB}
    I-P_{1}^{-1}B=R_0^T({B}_0)^{-1}R_0B-R_0^T({B}_{P_1})^{-1}R_0B.
\end{equation}

Let  $v_h=\begin{bmatrix}
  v_{\Gamma}    \\
    v_{I}     
\end{bmatrix}=\begin{bmatrix}v_\Gamma \\ \displaystyle{\sum_{i=1}^NR_{I_iI}^TQ_L^{(i)}\alpha_L^{(i)}+\sum_{i=1}^NR_{I_iI}^TQ_S^{(i)}\alpha_S^{(i)}}
\end{bmatrix}\in V_h$,  then we have the unique decomposition $v_h=R_0^T{v}_0+\displaystyle{\sum_{i=1}^N}R_i^Tv_i,$ where ${v}_0= \begin{bmatrix}
  v_{\Gamma}    \\
    \alpha_S
\end{bmatrix}\in {V}_0$  and 
\begin{equation*}
    v_i=B_L^{(i)}(B_{I\Gamma}^{(i)}R_{\Gamma_i\Gamma}v_\Gamma+B_{II}^{(i)}R_{I_iI}v_I)\in V_i \quad \text{ for }  1\leq i \leq N.
\end{equation*}

In order to bound $(I - P_1^{-1} B)(R_0^T v_0 + \sum_{i=1}^N R_i^T v_i)$, we first consider the bound for $(I - P_1^{-1} B)R_0^T v_0$. Using the property of $B_0$, we have:
\begin{equation*}
R_0^T({B}_0)^{-1}R_0B R_0^T{v}_0=R_0^T({B}_0)^{-1}B_0
{v}_0=R_0^T{v}_0,
\end{equation*}
and also 
\begin{equation*}
R_0^T({B}_{P_1})^{-1}R_0B R_0^T{v}_0=R_0^T({B}_{P_1})^{-1}B_0
{v}_0.
\end{equation*}

Together with \cref{I-PB}, \cref{diffb0bp1} and \cref{infsup_b_0}, we have that 
\begin{equation}
\label{H-R}
\begin{split}
||(I-P_{1}^{-1}B)R_0^T{v}_0||_{\mathcal{H}}&=||R_0^T{v}_0-R_0^T({B}_{P_1})^{-1}{B}_0{v}_0||_{\mathcal{H}}=||R_0^T({B}_{P_1})^{-1}(B_{P_1}-{B}_0){v}_0||_{\mathcal{H}}\\
&\leq\frac{2\eta}{\gamma_1}|||{v}_0|||\leq \frac{2\eta}{\gamma_1}||v_h||_{\mathcal{H}}.
\end{split} 
\end{equation}
  
So, it only reminds to bound $(I-P_{1}^{-1}B)\begin{bmatrix}0 \\ \displaystyle{\sum_{i=1}^NR_{I_iI}^Tv_i}
\end{bmatrix}$. Notice that   $R_0B\begin{bmatrix}0 \\ \displaystyle{\sum_{i=1}^NR_{I_iI}^Tv_i}
\end{bmatrix}=\vec{0},$  then we have:
\begin{equation*}
\begin{split}
(I-P_{1}^{-1}B)\begin{bmatrix}0 \\ \displaystyle{\sum_{i=1}^NR_{I_iI}^Tv_i}
\end{bmatrix}&=\vec{0}.
\end{split}
\end{equation*}

Finally, together with \cref{H-R}, we have that
$$||(I-P_{1}^{-1}B)v_h||_{\mathcal{H}}\leq \frac{2\eta}{\gamma_1}||v_h||_{\mathcal{H}}.$$

\end{proof}

Next, using \cref{I-P_B}, we obtain the following bound on the preconditioned system.
\begin{theorem}
\label{exact_theom}
For the upper bound of the preconditioned system, we have:
$$\max_{v_h\in V_h}\frac{||P_{1}^{-1}Bv_h||_{\mathcal{H}}}{||v_h||_{\mathcal{H}}}\leq 1+C(\eta,\gamma_1).$$
Moreover, when $C(\eta,\gamma_1)=\frac{2\eta}{\gamma_1}<1$, the lower bound of the preconditioned system satisfies:
$$ \min_{v_h\in V_h}\frac{|(v_h,P_{1}^{-1}Bv_h)_{\mathcal{H}}|}{||v_h||_{\mathcal{H}}^2}\geq 1-C(\eta,\gamma_1).$$
\end{theorem}
\begin{proof}
Notice that $||P_{1}^{-1}Bv_h||_{\mathcal{H}}\leq ||v_h||_{\mathcal{H}}+||(I-P_{1}^{-1}B)v_h||_{\mathcal{H}}$, the upper bound follows directly from \cref{I-P_B}.

For the lower bound, using the Cauchy–Schwarz inequality and \Cref{I-P_B}, we get:
$$ (v_h,(I-P_{1}^{-1}B)v_h)_{\mathcal{H}}\leq C(\eta,\gamma_1)||v_h||_\mathcal{H}^2.$$
Then, the lower bound is easily obtained by moving $C(\eta,\gamma_1)||v_h||_\mathcal{H}$ to the left-hand side and moving $(v_h,P_{1}^{-1}Bv_h)_{\mathcal{H}}$ to the right-hand side.
\end{proof}

\begin{remark}
All the above theorems hold for $P_2^{-1}$ with the constant $2$ replaced by $C_{max} = O(H/h)$, where $C_{max}$ is a constant such that
$C_0^{(i)} \leq C_{max} H_0^{(i)}$.
\end{remark}

\section{NOSAS Preconditioners for $\Re B_0$ and $\Im B_0$}
\label{section_5}
From \cref{remark4.4}, the scalability of $C^{-1}$ depends on the imaginary part of the linear system. For the Helmholtz equation subject to impedance boundary conditions, the imaginary part of the matrix $B_{\Gamma\Gamma}^{(i)}$ typically consists of a large block matrix, except in specific cases such as the HDG method with real penalty parameters, where it is already in block diagonal structure. To obtain better scalability, we propose a new preconditioner that separately handles both the real and imaginary parts, while still ensuring robust convergence for the preconditioned system. In this section, we focus primarily on the linear finite element space and present the corresponding results. However, the construction for other discretizations is similar, since all the construction is entirely algebraic, and the theoretical proofs are similar.

For the real part, we consider the construction for $b_{P_1}(\cdot, \cdot)$. Specifically, we consider generalized eigenvalue problems in \cref{gen_eig_R11} and \cref{gen_eig_R12}. We set a threshold $\eta_{Re} \in (0,1)$ and select eigenvalues that are smaller than $1-\eta_{Re}$ or greater than $1+\eta_{Re}$. We denote the corresponding eigenvector space as ${Q}_{Re}^{(i)} = [\xi_{Re_1}^{(i)}, \xi_{Re_2}^{(i)}, \cdots, \xi_{Re_{{K}_i}}^{(i)}]$, and we denote $\Pi_{Re}^{(i)} = Q_{Re}^{(i)} Q_{Re}^{(i)^T} H_0^{(i)}$.

To develop better parallelization properties, we also consider the imaginary part of $B_0^{(i)}$, which is exactly $\Im B_{\Gamma\Gamma}^{(i)}$ in the linear finite element space. Let us consider the following generalized eigenvalue problem:
\begin{equation}
\label{gen_eig_I1}
\Im B_{\Gamma\Gamma}^{(i)}\xi_{Im_j}^{(i)}={\lambda}_{Im_j}^{(i)} hk S_{\Pi\Pi}^{(i)}{\xi}_{Im_j}^{(i)}
\hspace{30pt} (j=1,\cdots, \hat{N}_i),
\end{equation}
where the eigenvectors $\xi_{Im_j}^{(i)}$ are orthonormal with respect to $S_{\Pi\Pi}^{(i)}$. Here $S_{\Pi\Pi}^{(i)} = H_{\Pi\Pi}^{(i)} - H_{\Pi R}^{(i)} (H_{RR}^{(i)})^{-1} H_{R\Pi}^{(i)}$ is the minimum $\mathcal{H}^{(i)}$-energy extension from $V_h(\Pi_i)$ to $V_h(\Omega_i)$, where $\Pi_i:=\partial\Omega_i\bigcap \partial\Omega$, $R_i:=I_i\bigcup (\partial\Omega_i\backslash \partial\Omega)$ and $\hat{N}_i$ is the number of DOFs on $\Gamma_i \cap \partial \Omega$. For a threshold $\eta_{Im} \in (0,1)$, we denote $\hat{Q}_1^{(i)} = [\xi_{Im_1}^{(i)}, \xi_{Im_2}^{(i)}, \cdots, \xi_{Im_{\hat{k}}}^{(i)}]$, where the chosen eigenfunctions correspond to the eigenvalues that are smaller than $1-\eta_{Im}$ or greater than $1+\eta_{Im}$. Let $\hat{Q}_{1}^{(i)^\perp} = [\xi_{Im_{\hat{k}+1}}^{(i)}, \cdots, \xi_{Im_{\hat{N}_i}}^{(i)}]$, which are the eigenfunctions corresponding to the eigenvalues between $1-\eta_{Im}$ and $1+\eta_{Im}$. 

We should also consider the following generalized eigenvalue problem:
\begin{equation}
\label{gen_eig_I2}
(\hat{Q}_{1}^{(i)^\perp})^T {C}_{\text{Diag}}^{(i)}\hat{Q}_{1}^{(i)^\perp}\phi^{(i)}_{Im_j}=\lambda^{(i)}_{Im_j} (\hat{Q}_{1}^{(i)^\perp})^T hkS_{\Pi\Pi}^{(i)}\hat{Q}_{1}^{(i)^\perp}\phi^{(i)}_{Im_j},
\end{equation}
where $C_{\text{Diag}}^{(i)}$ is the diagonal or block-diagonal form of $\Im B_{\Gamma\Gamma}^{(i)}$. We denote $\hat{Q}_2^{(i)} = [\hat{Q}_{1}^{(i)^\perp} \phi_{Im_1}^{(i)}, \hat{Q}_{1}^{(i)^\perp} \phi_{Im_2}^{(i)}, \cdots, \hat{Q}_{1}^{(i)^\perp} \phi_{Im_{\hat{s}}}^{(i)}]$, where the chosen eigenfunctions correspond to the eigenvalues that are smaller than $1-\eta_{Im}$ or greater than $1+\eta_{Im}$. We denote $Q_{Im}^{(i)} = [\hat{Q}_1^{(i)}, \hat{Q}_2^{(i)}]$, and $\Pi_{Im}^{(i)} = Q_{Im}^{(i)} Q_{Im}^{(i)^T} S_{\Pi\Pi}^{(i)}$.

\begin{remark}
On the right-hand side of the generalized eigenvalue problems in \cref{gen_eig_I1} and \cref{gen_eig_I2}, there is a constant $hk$ on the right-hand side. This is because in the linear finite element space we have $\Im B_{\Gamma\Gamma}^{(i)} \leq C_h hk A_{\Pi\Pi}^{(i)}$, where $A_{\Pi\Pi}^{(i)}$ is the $\Pi_i$ part of the stiffness matrix $A^{(i)}$, and $C_h = O(1)$ is a constant that depends on the shape of the elements. Therefore, we have:
\begin{equation}
 \Im B_{\Gamma\Gamma}^{(i)}\leq C_hkh H_{\Pi\Pi}^{(i)} \leq C_hC_{max}hkS_{\Pi\Pi}^{(i)},
\end{equation}
where $C_{max}=O(H/h)$ is the constant such that $ H_{\Pi\Pi}^{(i)}\leq C_{max}S_{\Pi\Pi}^{(i)}.$
\end{remark}

We then define a sesquilinear form ${b}_{P_3}(\cdot, \cdot): V_0 \times V_0 \to \mathbb{C}$ as follows:
\begin{equation*}
\label{sesq_bp3}
\begin{split}
    {b}_{P_3}^{(i)}(u_{0}^{(i)},v_{0}^{(i)})&=v_{0}^{(i)^H}\Big(\Re B_0^{(i)}+\textbf{i}\Im B_{0}^{(i)}-(I-\Pi_{Re}^{(i)^T})\Re B_0^{(i)}(I-\Pi_{Re}^{(i)})+(I-\Pi_{Re}^{(i)^T})C_{0}^{(i)}(I-\Pi_{Re}^{(i)})\cdots\\
    &-(I-\Pi_{Im}^{(i)^T})\textbf{i}\Im B_0^{(i)}(I-\Pi_{Im}^{(i)})+(I-\Pi_{Im}^{(i)^T})\textbf{i}C_\text{Diag}^{(i)}(I-\Pi_{Im}^{(i)})\Big)u_{0}^{(i)}\\
&=v_{0}^{(i)^H}\Big(C_0^{(i)}+\textbf{i}C_\text{Diag}^{(i)}-D_{Re}^{(i)}\Pi_{Re}^{(i)}-\Pi_{Re}^{(i)^T} D_{Re}^{(i)}(I-\Pi_{Re}^{(i)})-\textbf{i}D_{Im}^{(i)}\Pi_{Im}^{(i)}\cdots\\&-\textbf{i}\Pi_{Im}^{(i)^T} D_{Im}^{(i)}(I-\Pi_{Im}^{(i)})\Big)u_{0}^{(i)},
   \end{split}
\end{equation*}
where  $D_{Re}^{(i)}=C_0^{(i)}-\Re B_0^{(i)}$, and  $D_{Im}^{(i)}=C_\text{Diag}^{(i)}-\Im B_0^{(i)}$.

The resulting preconditioner can be written as:
\begin{equation*}
\label{P_ad3}
P_{3}^{-1}=R_0^T(B_{P_3})^{-1}R_0+\displaystyle{\sum_{i=1}^N}{R}_i^TQ_P^{(i)}(B_i)^{-1}Q_P^{(i)^T}{R}_i.
\end{equation*}

Similarly to the analysis in \Cref{section_4}, we have the following theorems for the well-posedness of the coarse problem of the second preconditioner.
\begin{theorem}
\label{inf_sup_tilde_b1}
Choosing $\eta$  such that $\gamma_2=\gamma-2\eta_{Re}-2hk\eta_{Im}>0$, the global sesquilinear form $b_{P_3}(\cdot,\cdot)$ satisfy the following inf-sup condition:
\begin{equation*}
 \inf_{{u}_0\in {V}_0\backslash\{0\}}\sup_{{v}_0\in {V}_0\backslash\{0\}}\frac{|b_{P_3}({u}_0,{v}_0)|}{|||{u}_0|||\hspace{3pt}|||{v}_0|||}\geq \gamma_2.
\end{equation*} 
\end{theorem}
\begin{proof}
     This follows from a similar proof as shown in \Cref{inf_sup_tilde_b0}.
\end{proof}

We also have the following theorem regarding the convergence rate of the preconditioner.
\begin{theorem}
For the upper bound of the preconditioned system, we have:
$$\max_{v_h\in V_h}\frac{||P_{3}^{-1}Bv_h||_{\mathcal{H}}}{||v_h||_{\mathcal{H}}}\leq 1+C(\eta_{Re},\eta_{Im},\gamma_2),$$
where $C(\eta_{Re},\eta_{Im},\gamma_2)=\frac{2\eta_{Re}+2hk\eta_{Im}}{\gamma_2}$. Moreover, if we choose $\eta_{Re}$ and $\eta_{Im}$ such that 
$C(\eta_{Re},\eta_{Im},\gamma_2)<1.$
Then, the lower bound of the preconditioned system satisfies:
$$ \min_{v_h\in V_h}\frac{|(v_h,P_{3}^{-1}Bv_h)_{\mathcal{H}}|}{||v_h||_{\mathcal{H}}^2}\geq 1-C(\eta_{Re},\eta_{Im},\gamma_2).$$
\end{theorem}
\begin{proof}
     This follows from a similar proof as shown in \Cref{exact_theom}.
\end{proof}

\begin{remark}
For the imaginary part, since $hk$ is very small, $\eta_{Im}$ can be chosen as a small number compared to $\eta_{Re}$. This is verified by numerical experiments, indicating that convergence remains robust with relatively smaller thresholds for the imaginary component.
\end{remark}

We also propose a cheaper preconditioning method similar to $P_2^{-1}$ in \cref{P_ad2}. Let $Q_{Re}^{(i)}$ and the projection operators $\Pi_{Re}^{(i)} = Q_{Re}^{(i)} Q_{Re}^{(i)^T} C_{0}^{(i)}$ be obtained from \cref{gen_eig_R}. We solve the following generalized eigenvalue problems locally for the imaginary part:
\begin{equation}
\label{gen_eig_I}
\Im B_0^{(i)}\xi^{(i)}_j=\lambda^{(i)}_j C_\text{Diag}^{(i)}\xi^{(i)}_j
\hspace{30pt} (j=1,\cdots, n_i),
\end{equation}
where the eigenvectors $\xi^{(i)}_j$ are orthonormal with respect to $C_\text{Diag}^{(i)}$. We denote $Q_{Im}^{(i)}$ as the eigenspace of eigenvalues that are either smaller than $1-\eta_{Im}$ or larger than $1+\eta_{Im}$. The projection operators are defined as $\Pi_{Im}^{(i)} = Q_{Im}^{(i)} Q_{Im}^{(i)^T} C_\text{Diag}^{(i)}$.

Next, we define the local sesquilinear form ${b}_{P_4}^{(i)}(\cdot,\cdot):V_0^{(i)}\times V_0^{(i)}\to \mathbb{C}$ as follows:
\begin{equation*}
    \begin{split}
     {b}_{P_4}^{(i)}(u_{0}^{(i)},v_{0}^{(i)})&=v_{0}^{(i)^H}\big(\Pi_{Re}^{(i)^T}\Re B_0^{(i)}\Pi_{Re}^{(i)}+(I-\Pi_{Re}^{(i)^T})C_{0}^{(i)}(I-\Pi_{Re}^{(i)})\cdots\\
     &+\Pi_{Im}^{(i)^T}\textbf{i}\Im B_0^{(i)}\Pi_{Im}^{(i)}+(I-\Pi_{Im}^{(i)^T})\textbf{i}C_\text{Diag}^{(i)}(I-\Pi_{Im}^{(i)}) \big)u_{0}^{(i)}\\
 &=v_{0}^{(i)^H}\big(C^{(i)}_{0}+\textbf{i}C_\text{Diag}^{(i)}\!-C_{0}^{(i)}Q_{Re}^{(i)}D_{Re}^{(i)}Q_{Re}^{(i)^T}\!\!C_{0}^{(i)}-\textbf{i}C_\text{Diag}^{(i)}Q_{Im}^{(i)}D_{Im}^{(i)}Q_{Im}^{(i)^T}\!\!C_\text{Diag}^{(i)}\big)u_{0}^{(i)},
\end{split}
\end{equation*}
where $D_{Re}^{(i)} = \text{diag}(1 - \lambda_{Re_1}^{(i)}, \cdots, 1 - \lambda_{Re_{K_i}}^{(i)})$, and $D_{Im}^{(i)} = \text{diag}(1 - \lambda_{Im_1}^{(i)}, \cdots, 1 - \lambda_{Im_{\hat{K}_i}}^{(i)})$. The global sesquilinear form ${b}_{P_4}(\cdot, \cdot)$ is then defined as the sum of all local sesquilinear forms. The resulting preconditioner is:
\begin{equation*}
\label{P_ad4}
P_{4}^{-1}=R_0^T(B_{P_4})^{-1}R_0+\displaystyle{\sum_{i=1}^N}{R}_i^TQ_P^{(i)}(B_i)^{-1}Q_P^{(i)^T}{R}_i.
\end{equation*}

\begin{remark}
All of the above theorems also hold for $P_4^{-1}$, with the constant $2$ replaced by $C_{max} = O(H/h)$.
\end{remark}

\section{Numerical Experiments}
\label{section_6}

We present numerical results for problem \cref{hel_cont} with $f = 0$ and varying wavenumbers $k$. The computational domain is the square $\Omega = (0,1)^2$, with an impedance boundary condition defined by the plane wave $e^{\textbf{i}k\vec{V} \cdot \vec{X}}$ on $\partial \Omega$. Here, the direction vector is $\vec{V} = \langle \cos(\pi/8), \sin(\pi/8) \rangle$, and $\vec{X}$ represents the coordinate vector of the wave. The domain $\Omega$ is discretized using a triangulation $\mathcal{T}_h$, which consists of $1/h^2$ congruent squares, each subdivided into two right-angled triangular elements. Additionally, we partition the square domain into $1/H^2$ congruent square subdomains, where $H$ is an integer multiple of $h$.

To evaluate the performance of the four preconditioners, we employ different finite element spaces, focusing primarily on $P_{2}^{-1}$ and $P_{4}^{-1}$ due to their lower computational cost. For all tests, we set $\beta = 0.01$ in the generalized eigenvalue problems \cref{gen_eig_R00} to ensure local problem solvability. For the real part generalized eigenvalue problems \cref{gen_eig_R}, we test performance with varying $\eta_{Re}$ and select eigenvalues outside $(1-\eta_{Re},1+\eta_{Re})$ to construct the coarse problem. For $P_{4}^{-1}$, we additionally consider the imaginary part generalized eigenvalue problems \cref{gen_eig_I} with $\eta_{Im}=0.9$, but no eigenfunctions are selected since all eigenvalues lie within $(0.1,1.9)$. Therefore, all tests primarily display the iteration counts and selected eigenfunctions from the real part eigenvalue problem. The method is implemented via \cref{Alg_1}, where $N$ local problems are solved both pre- and post-iteration, and GMRES solves the preconditioned system with $B_0$ until the relative residual error falls below $10^{-6}$ in the $\ell^2$-norm.

We first present numerical results using the linear finite element method, comparing the performance of preconditioners $P_{1}^{-1}$ and $P_{4}^{-1}$. The results, summarized in \cref{Table_1} and \cref{Table_2}, indicate that $P_{1}^{-1}$ selects a smaller number of eigenvalues and achieves a slightly better convergence rate, consistent with the convergence theorem. However, the global matrix size for $P_{1}^{-1}$ is twice the number of selected eigenfunctions, while for $P_{4}^{-1}$ it equals the number of selected eigenfunctions, therefore the two preconditioners have similar sizes of global matrices. Considering $P_{4}^{-1}$'s superior scalability and lower computational cost for solving the eigenvalue problem, we conclude that $P_{4}^{-1}$ is the more effective preconditioner overall.

\begin{table}[tbhp]
    \centering
    \small
    \begin{minipage}{0.48\textwidth}
        \centering
        \begin{tabular}{|c|c|c|c|c|}
        \hline
        \multicolumn{5}{|c|}{$k=20$} \\
        \hline
        \diagbox[width=3em]{$h$}{$H$} & $1/2$ & $1/4$ & $1/8$ & $1/16$ \\ 
        \hline
        $1/32$  & 8(18) & 8(9) & 7(4) & 6(3) \\ 
        \hline
        $1/64$  & 7(22) & 7(12) & 6(8) & 7(4) \\ 
        \hline
        $1/128$ & 7(22) & 7(12) & 7(8) & 8(5) \\ 
        \hline
        \end{tabular}
    \end{minipage}
    \hfill
    \begin{minipage}{0.48\textwidth}
        \centering
        \begin{tabular}{|c|c|c|c|c|}
        \hline
        \multicolumn{5}{|c|}{$k=30$} \\
        \hline
        \diagbox[width=3em]{$h$}{$H$} & $1/4$ & $1/8$ & $1/16$ & $1/32$ \\ 
        \hline
        $1/64$  & 7(16) & 6(9) & 7(4) & 6(3) \\ 
        \hline
        $1/128$ & 7(16) & 8(9) & 7(7) & 8(4) \\ 
        \hline
        $1/256$ & 8(16) & 9(9) & 8(7) & 9(5) \\ 
        \hline
        \end{tabular}
    \end{minipage}

    \vspace{1em}
    \begin{minipage}{0.7\textwidth}
        \centering
        \begin{tabular}{|c|c|c|c|c|}
        \hline
        \multicolumn{5}{|c|}{$k=40$} \\
        \hline
        \diagbox[width=3em]{$h$}{$H$} & $1/8$ & $1/16$ & $1/32$ & $1/64$ \\ 
        \hline
        $1/128$ & 7(12) & 7(8) & 9(4) & 6(3) \\ 
        \hline
        $1/256$ & 7(12) & 8(8) & 10(5) & 9(4) \\ 
        \hline
        $1/512$ & 8(12) & 8(8) & 14(5) & 11(5) \\ 
        \hline
        \end{tabular}
    \end{minipage}
 \vspace{1em}
    \caption{GMRES iterations with preconditioner $P_1^{-1}$ ($\beta=0.01$, $\eta=0.4$) using P1 finite element. Parentheses indicate the number of selected eigenfunctions per subdomain for the generalized eigenvalue problems \cref{gen_eig_R11} and \cref{gen_eig_R12}.}
    \label{Table_1}
\end{table}

\begin{table}[tbhp]
    \centering
    \small
    \begin{minipage}{0.48\textwidth}
        \centering
        \begin{tabular}{|c|c|c|c|c|}
        \hline
        \multicolumn{5}{|c|}{$k=20$} \\
        \hline
        \diagbox[width=3em]{$h$}{$H$} & $1/2$ & $1/4$ & $1/8$ & $1/16$ \\ 
        \hline
        $1/32$  & 10(17) & 10(8) & 9(4) & 8(6) \\ 
        \hline
        $1/64$  & 11(25) & 11(13) & 11(9) & 10(7) \\ 
        \hline
        $1/128$ & 11(27) & 10(17) & 10(12) & 10(10) \\ 
        \hline
        \end{tabular}
    \end{minipage}
    \hfill
    \begin{minipage}{0.48\textwidth}
        \centering
        \begin{tabular}{|c|c|c|c|c|}
        \hline
        \multicolumn{5}{|c|}{$k=30$} \\
        \hline
        \diagbox[width=3em]{$h$}{$H$} & $1/4$ & $1/8$ & $1/16$ & $1/32$ \\ 
        \hline
        $1/64$  & 11(13) & 15(7) & 9(6) & 8(6) \\ 
        \hline
        $1/128$ & 10(19) & 14(12) & 12(9) & 11(8) \\ 
        \hline
        $1/256$ & 10(24) & 16(16) & 13(13) & 13(9) \\ 
        \hline
        \end{tabular}
    \end{minipage}

    \vspace{1em}
    \begin{minipage}{0.7\textwidth}
        \centering
        \begin{tabular}{|c|c|c|c|c|}
        \hline
        \multicolumn{5}{|c|}{$k=40$} \\
        \hline
        \diagbox[width=3em]{$h$}{$H$} & $1/8$ & $1/16$ & $1/32$ & $1/64$ \\ 
        \hline
        $1/128$ & 11(13) & 12(9) & 11(7) & 8(6) \\ 
        \hline
        $1/256$ & 10(17) & 10(12) & 12(10) & 12 (9) \\ 
        \hline
        $1/512$ & 10(20) & 9(16) & 13(13) & 10(11) \\ 
        \hline
        \end{tabular}
    \end{minipage}
 \vspace{1em}
    \caption{GMRES iterations with preconditioner $P_4^{-1}$ ($\beta=0.01$, $\eta_{Re}=0.4$, $\eta_{Im}=0.9$) using P1 finite element. Parentheses indicate the number of selected eigenfunctions per subdomain for the generalized eigenvalue problems \cref{gen_eig_R}.}
    \label{Table_2}
\end{table}

As the wavenumber $k$ increases, the mesh size $h$ must decrease to maintain proper resolution for the Helmholtz equation, consequently increasing iteration counts. We evaluate $P_{4}^{-1}$'s performance for large wavenumbers using both P1 and P2 finite elements in \cref{Table_3} and \cref{Table_4}, respectively. Our numerical experiments show that reducing the threshold parameter $\eta$ effectively mitigates the iteration growth at higher wavenumbers.  For fixed values of $k$ and $h$, reducing the subdomain size $H$ leads to larger global problem dimensions, as our approximation strategy for near-zero eigenvalues becomes less effective with decreasing $H$. Thus, optimal selection of both $H$ and $\eta$ is crucial for Helmholtz preconditioning. The optimal choice of $H$ depends on computational constraints: the processing capacity per compute node and the maximum locally solvable eigenvalue problem size. The threshold $\eta$ is determined empirically. The rule of thumb for choosing an appropriate threshold is to select as few eigenfunctions as possible while ensuring reasonably fast convergence. We avoid selecting an excessively small threshold, as this would include too many eigenfunctions, thereby increasing computational costs despite potentially reducing iteration counts. Furthermore, as demonstrated by the convergence theorem, the impact of the imaginary part is minimal - even when $\eta_{Im}$ decreases to $0.5$, only one additional eigenvalue is chosen per subdomain, with negligible effect on the convergence rate. Since $\eta_{Im}=0.9$ selects no eigenfunctions for the imaginary part generalized eigenvalue problems, the global problem is constructed solely from eigenfunctions selected in the real part. Finally, we note that selecting all eigenvalues from the real and imaginary parts yields  a direct solver.

\begin{table}[tbhp]
    \centering
    \small
    \begin{minipage}{0.48\textwidth}
        \centering
        \begin{tabular}{|c|c|c|c|c|}
        \hline
        \multicolumn{5}{|c|}{$k=50$} \\
        \hline
        \diagbox[width=3em]{$h$}{$H$} & $1/8$ & $1/16$ & $1/32$ & $1/64$ \\ 
        \hline
        $1/256$  & 9(27) & 9(17) & 11(13) & 8(11) \\ 
        \hline
        $1/512$  & 9(28) & 9(20) & 9(16) & 10(13) \\ 
        \hline
        $1/1024$ & 9(29) & 9(21) & 9(17) & 9(15) \\ 
        \hline
        \end{tabular}
    \end{minipage}
 \hfill
    \begin{minipage}{0.48\textwidth}
        \centering
        \begin{tabular}{|c|c|c|c|c|}
        \hline
        \multicolumn{5}{|c|}{$k=60$} \\
        \hline
        \diagbox[width=3em]{$h$}{$H$} & $1/16$ & $1/32$ & $1/64$ & $1/128$ \\ 
        \hline
        $1/512$  & 9(20) & 9(16) & 10(13) & 9(11) \\ 
        \hline
        $1/1024$ & 9(21) & 10(16) & 9(16) & 10(13) \\ 
        \hline
        $1/2048$ & 9(27) & 9(23) & 10(17) & 9(15) \\ 
        \hline
        \end{tabular}
    \end{minipage}
 \vspace{1em}
    \caption{GMRES iterations with preconditioner $P_4^{-1}$ ($\beta=0.01$, $\eta_{Re}=0.2$, $\eta_{Im}=0.9$) using P1 finite element. Parentheses indicate the number of selected eigenfunctions per subdomain for the generalized eigenvalue problems \cref{gen_eig_R}.}
    \label{Table_3}
\end{table}

\begin{table}[tbhp]
    \centering
    \small
    \begin{minipage}{0.48\textwidth}
        \centering
        \begin{tabular}{|c|c|c|c|c|}
        \hline
        \multicolumn{5}{|c|}{$k=50$} \\
        \hline
        \diagbox[width=3em]{$h$}{$H$} & $1/8$ & $1/16$ & $1/32$ & $1/64$ \\ 
        \hline
        $1/128$  &  7(28) &  6(20) & 7(15) &  7(12) \\ 
        \hline
        $1/256$  &  7(28) &  7(20) & 7(16) & 10(13) \\ 
        \hline
        $1/512$ &  7(29) &  7(21) &  8(18) &  7(16) \\ 
        \hline
        \end{tabular}
    \end{minipage}
 \hfill
    \begin{minipage}{0.48\textwidth}
        \centering
        \begin{tabular}{|c|c|c|c|c|}
        \hline
        \multicolumn{5}{|c|}{$k=60$} \\
        \hline
        \diagbox[width=3em]{$h$}{$H$} & $1/16$ & $1/32$ & $1/64$ & $1/128$ \\ 
        \hline
        $1/256$  &  7(20) &  7(16) &  10(13) &  7(12) \\ 
        \hline
        $1/512$ &  7(23) &  8(17) &  7(16) &  10(13) \\ 
        \hline
        $1/1024$ & 7(28) & 7(23) & 7(20) &  8(16) \\ 
        \hline
        \end{tabular}
    \end{minipage}
 \vspace{1em}
    \caption{GMRES iterations with preconditioner $P_4^{-1}$ ($\beta=0.01$, $\eta_{Re}=0.2$, $\eta_{Im}=0.9$) using P2 finite element. Parentheses indicate the number of selected eigenfunctions per subdomain for the generalized eigenvalue problems \cref{gen_eig_R}.}
    \label{Table_4}
\end{table}

We next present results for the linear HDG method. For implementation details of the HDG method, we refer to \cite{cockburn2009unified}. A crucial aspect of HDG for the Helmholtz equation is the selection of the stabilization parameter $\tau$, which may be either purely real or purely imaginary. Stability and error analyses for these choices are available in \cite{zhu2021preasymptotic, cui2014analysis}. When $\tau$ is purely real, the imaginary part of the resulting HDG linear system becomes block diagonal, with each block corresponding to an element. This structure allows direct application of $P_{2}^{-1}$ without considering the imaginary part generalized eigenvalue problems. Conversely, when $\tau$ is purely imaginary, the imaginary part of the linear system forms a globally assembled matrix, requiring the use of $P_{4}^{-1}$. In this case, since the local problem is well-posed, we avoid solving the first generalized eigenvalue problem, resulting in the simplified local solver $B_i = R_iBR_i^T = B_{II}^{(i)}$.

In the numerical tests below, we evaluate the performance of these two preconditioners for different choices of $\tau$. \Cref{Table_5} and \Cref{Table_6} present the results for real $\tau$ with varying thresholds $\eta$, while \Cref{Table_7} shows the corresponding results for imaginary $\tau$. The preconditioner performs significantly better for real $\tau$ than for imaginary $\tau$, since for real $\tau$ we only need to solve real eigenvalue problems, and the imaginary part remains block diagonal due to HDG properties - making the preconditioner's imaginary part identical to the original system's imaginary component. Similar to FEM cases, our experiments confirm that increasing wavenumber $k$ and decreasing mesh size $h$ requires reducing threshold parameter $\eta$ to maintain convergence rates. For fixed $k$ and $h$, smaller subdomain size $H$ increases coarse problem dimensions due to our preconditioner's approximation properties, making optimal selection of both $H$ and $\eta$ crucial for Helmholtz preconditioning. Theoretically, we select eigenvalues either smaller than $1-\eta$ or larger than $1+\eta$. Numerically, we observe that for Helmholtz problems, larger eigenvalues significantly degrade convergence rates - unlike elliptic problems where large eigenvalues have a modest impact on the convergence rate. Furthermore, we find that selecting eigenvalues either smaller than $1-\eta$ or larger than 2 achieves comparable convergence rates to those shown in the tables, while substantially reducing the number of selected eigenfunctions per subdomain and consequently decreasing the global problem size.

\begin{table}[tbhp]
    \centering
    \small
    \begin{minipage}{0.48\textwidth}
        \centering
        \begin{tabular}{|c|c|c|c|c|}
        \hline
        \multicolumn{5}{|c|}{$k=20$} \\
        \hline
        \diagbox[width=3em]{$h$}{$H$} & $1/2$ & $1/4$ & $1/8$ & $1/16$ \\ 
        \hline
        $1/32$  & 18(17) & 14(9) & 13(6) & 19(4) \\ 
        \hline
        $1/64$  & 11(30) & 14(13) & 12(9) & 11(6) \\ 
        \hline
        \end{tabular}
    \end{minipage}
    \hfill
    \begin{minipage}{0.48\textwidth}
        \centering
        \begin{tabular}{|c|c|c|c|c|}
        \hline
        \multicolumn{5}{|c|}{$k=30$} \\
        \hline
        \diagbox[width=3em]{$h$}{$H$} & $1/4$ & $1/8$ & $1/16$ & $1/32$ \\ 
        \hline
        $1/64$  & 12(16) & 12(10) & 13(6) & 21(4) \\ 
        \hline
        $1/128$ & 13(24) & 12(14) & 12(8) & 11(6) \\ 
        \hline
        \end{tabular}
    \end{minipage}

    \vspace{1em}
    \begin{minipage}{0.48\textwidth}
        \centering
        \begin{tabular}{|c|c|c|c|c|}
        \hline
        \multicolumn{5}{|c|}{$k=40$} \\
        \hline
        \diagbox[width=3em]{$h$}{$H$} & $1/8$ & $1/16$ & $1/32$ & $1/64$ \\ 
        \hline
        $1/128$ & 15(13) & 13(9) & 13(6) & 22(4) \\ 
        \hline
        $1/256$ &  18(25) &  11(14) & 14(7) & 13(6) \\ 
        \hline
        \end{tabular}
    \end{minipage}
     \hfill
      \begin{minipage}{0.48\textwidth}
        \centering
        \begin{tabular}{|c|c|c|c|c|}
        \hline
        \multicolumn{5}{|c|}{$k=50$} \\
        \hline
        \diagbox[width=3em]{$h$}{$H$} & $1/16$ & $1/32$ & $1/64$ & $1/128$ \\ 
        \hline
        $1/256$ & 12(14) & 17(7) & 14(6) & 23(4) \\ 
        \hline
        $1/512$ &  10(25) &  11(14) & 13(7) & 14(6) \\ 
        \hline
        \end{tabular}
    \end{minipage}
     \vspace{1em}
    \caption{GMRES iterations with preconditioner $P_2^{-1}$ ($\beta=0.01$, $\eta=0.7$) using linear HDG. Parentheses indicate the number of selected eigenfunctions per subdomain for the generalized eigenvalue problems \cref{gen_eig_R}.}
    \label{Table_5}
\end{table}

\begin{table}[tbhp]
    \centering
    \small
    \begin{minipage}{0.48\textwidth}
        \centering
        \begin{tabular}{|c|c|c|c|c|}
        \hline
        \multicolumn{5}{|c|}{$k=20$} \\
        \hline
        \diagbox[width=3em]{$h$}{$H$} & $1/2$ & $1/4$ & $1/8$ & $1/16$ \\ 
        \hline
        $1/32$  & 7(53) & 8(24) & 7(13) & 7(6) \\ 
        \hline
        $1/64$  & 6(105) & 7(51) & 7(25) & 6(12) \\ 
        \hline
        \end{tabular}
    \end{minipage}
    \hfill
    \begin{minipage}{0.48\textwidth}
        \centering
        \begin{tabular}{|c|c|c|c|c|}
        \hline
        \multicolumn{5}{|c|}{$k=30$} \\
        \hline
        \diagbox[width=3em]{$h$}{$H$} & $1/4$ & $1/8$ & $1/16$ & $1/32$ \\ 
        \hline
        $1/64$  & 7(50) & 7(25) & 7(13) & 20(5) \\ 
        \hline
        $1/128$ & 7(103) & 7(50) & 7(25) & 6(12) \\ 
        \hline
        \end{tabular}
    \end{minipage}

    \vspace{1em}
    \begin{minipage}{0.48\textwidth}
        \centering
        \begin{tabular}{|c|c|c|c|c|}
        \hline
        \multicolumn{5}{|c|}{$k=40$} \\
        \hline
        \diagbox[width=3em]{$h$}{$H$} & $1/8$ & $1/16$ & $1/32$ & $1/64$ \\ 
        \hline
        $1/128$ &  7(51) & 7(25) & 7(12) & 21(6) \\ 
        \hline
        $1/256$ & 6(102) & 6(51) & 6(26) & 6(12) \\ 
        \hline
        \end{tabular}
    \end{minipage}
     \hfill
      \begin{minipage}{0.48\textwidth}
        \centering
        \begin{tabular}{|c|c|c|c|c|}
        \hline
        \multicolumn{5}{|c|}{$k=50$} \\
        \hline
        \diagbox[width=3em]{$h$}{$H$} & $1/16$ & $1/32$ & $1/64$ & $1/128$ \\ 
        \hline
        $1/256$ &  7(50) & 7(26) & 6(12) & 23(6) \\
        \hline
        $1/512$ &  6(102) & 6(51) & 6(26) & 6(12) \\
        \hline
        \end{tabular}
    \end{minipage}
     \vspace{1em}
    \caption{GMRES iterations with preconditioner $P_2^{-1}$ ($\beta=0.01$, $\eta=0.4$) using linear HDG with $\tau_K=
    k$. Parentheses indicate the number of selected eigenfunctions per subdomain for the generalized eigenvalue problems \cref{gen_eig_R}.}
    \label{Table_6}
\end{table}

\begin{table}[tbhp]
    \centering
    \small
    \setlength{\tabcolsep}{3pt} 
    \begin{minipage}{0.48\textwidth}
        \centering
        \begin{tabular}{|c|c|c|c|c|}
        \hline
        \multicolumn{5}{|c|}{$k=20$} \\
        \hline
        \diagbox[width=2.5em]{$h$}{$H$} & $1/2$ & $1/4$ & $1/8$ & $1/16$ \\ 
        \hline
        $1/32$  &  15(17,14) & 14(9,8) & 14(6,5) & 19(4,2) \\ 
        \hline
        $1/64$  &  14(27,27) & 14(13,18) & 14(9,9) & 13(6,4) \\ 
        \hline
        \end{tabular}
    \end{minipage}
    \hfill
    \begin{minipage}{0.48\textwidth}
        \centering
        \begin{tabular}{|c|c|c|c|c|}
        \hline
        \multicolumn{5}{|c|}{$k=30$} \\
        \hline
        \diagbox[width=2.5em]{$h$}{$H$} & $1/4$ & $1/8$ & $1/16$ & $1/32$ \\ 
        \hline
        $1/64$  &  14(15,16) & 12(10,9) & 14(6,5) & 25(4,2)\\ 
        \hline
        $1/128$ &  15(24,36) & 14(14,18) & 15(8,9) & 16(6,4) \\ 
        \hline
        \end{tabular}
    \end{minipage}

    \vspace{6pt}
    \begin{minipage}{0.48\textwidth}
        \centering
        \begin{tabular}{|c|c|c|c|c|}
        \hline
        \multicolumn{5}{|c|}{$k=40$} \\
        \hline
        \diagbox[width=2.5em]{$h$}{$H$} & $1/8$ & $1/16$ & $1/32$ & $1/64$ \\ 
        \hline
        $1/128$ &  15(13,18) & 15(9,9) & 17(6,4) & 32(4,2) \\ 
        \hline
        $1/256$ &  21(25,36) & 15(14,18) & 21(7,9) & 22(6,4) \\ 
        \hline
        \end{tabular}
    \end{minipage}
    \hfill
    \begin{minipage}{0.48\textwidth}
        \centering
        \begin{tabular}{|c|c|c|c|c|}
        \hline
        \multicolumn{5}{|c|}{$k=50$} \\
        \hline
        \diagbox[width=2.5em]{$h$}{$H$} & $1/16$ & $1/32$ & $1/64$ & $1/128$ \\ 
        \hline
        $1/256$ &  15(14,18) & 24(7,9) & 23(6,4) & 44(4,2) \\ 
        \hline
        $1/512$ & 16(25,36) & 17(14,18) & 24(7,9) & 33(6,5) \\ 
        \hline
        \end{tabular}
    \end{minipage}
    
    \vspace{6pt}
    \caption{GMRES iterations with preconditioner $P_4^{-1}$ ($\beta=0.01$, $\eta_{Re}=0.7$, $\eta_{Im}=0.9$) using linear HDG with $\tau_K=\frac{\textbf{i}}{h}$. Parentheses indicate the number of selected eigenfunctions per subdomain for the generalized eigenvalue problem \cref{gen_eig_R} and \cref{gen_eig_I}, respectively.}
    \label{Table_7}
\end{table}

\section{Conclusion}
\label{sec:conclusion}

In this paper, we introduce a novel iterative substructuring approach that guarantees the well-posedness of local Dirichlet problems for arbitrary wavenumbers. The proposed substructuring framework provides a general structure that can be readily adopted by other domain decomposition methods. Building on this structure and motivated by the success of generalized eigenvalue problems in preconditioning elliptic equations with heterogeneous coefficients, we develop two new classes of preconditioners within the NOSAS framework: one targeting the real part of the Helmholtz operator and another jointly addressing its real and imaginary components. We establish asymptotic convergence theorems for near-zero thresholds and present supporting numerical results for various discretizations with large wavenumbers.

Our preconditioners possess several features:
\begin{enumerate}
    \item They are of nonoverlapping type, offering lower computational costs and memory requirements compared to overlapping methods.
    \item The generalized eigenvalue problems involve only real matrices with symmetric positive-definite right-hand sides, enabling efficient computation via Cholesky decomposition while ensuring high numerical accuracy in local eigenvalue and eigenvector computations.
    \item The construction is purely algebraic, based solely on local Neumann matrices, which allows for easy adaptation to various discretizations, heterogeneous coefficients, and complex geometries without geometric constraints.
\end{enumerate}

Future research directions include extensions to multilevel preconditioners and applications to related wave-propagation problems. The generality of our substructuring framework suggests potential adaptations to other domain decomposition methods, particularly for high-frequency or heterogeneous problems.

\section{Appendix}
\begin{theorem}[Sherman-Morrison-Woodbury formula]
\label{woodbury}
Let $A$ be an $n \times n$ matrix, $C_i$ be a $k_i \times k_i$ matrix, $U_i$ be an $n \times k_i$ matrix, and $V_i$ be a $k_i \times n$ matrix for $i = 1, 2$. Define $U = [U_1, U_2]$, $V = \begin{bmatrix} V_1 \\ V_2 \end{bmatrix}$, and $C = \begin{bmatrix} C_1 & \\ & C_2 \end{bmatrix}$. Assume that $A$, $C$, and $C^{-1} + VA^{-1}U$ are invertible. Then, $A + U_1C_1V_1 + U_2C_2V_2$ is also invertible, and we have that 
\begin{equation}
    (A + U_1C_1V_1 + U_2C_2V_2)^{-1} = A^{-1} - A^{-1}U(C^{-1} + VA^{-1}U)^{-1}VA^{-1}.
\end{equation}
\end{theorem}
\begin{proof}
Note that
$UCV = U_1C_1V_1 + U_2C_2V_2$. Therefore, we can apply the Sherman-Morrison-Woodbury formula to $A + UCV$.
\end{proof}
We note that if $A$, $C$, and $A + UCV$ are invertible, then $C^{-1} + VA^{-1}U$ must also be invertible.

\bibliographystyle{siamplain}
\bibliography{references}

\end{document}